\documentclass{amsart}
\usepackage[utf8]{inputenc}
\usepackage{a4wide}
\usepackage{amsmath,amsfonts,amssymb}
\usepackage{amsthm}
\usepackage{mathtools}
\usepackage{bbm}
\usepackage{ifthen}
\usepackage{longtable}
\usepackage{array}
\usepackage{url}
\usepackage{enumerate}
\usepackage{enumitem}
\usepackage{xcolor}
\usepackage{float}
\usepackage{graphicx}
\usepackage{todonotes}
\usepackage{underscore} 
\usepackage{hyperref}
\usepackage{cleveref}

\usepackage{titlesec}
\titleformat{\section}
  {\center\normalsize\bfseries}{\thesection.}{1em}{}
\titleformat{\subsection}
  {\normalsize\bfseries}{\thesubsection.}{1ex}{}

\hypersetup{colorlinks,
	linkcolor=green!50!black,
	citecolor=blue!70!black,
	urlcolor=red!50!black
}

\setlist{topsep = 4pt, itemsep=4pt} 

\theoremstyle{plain}
    \newtheorem{thm}{Theorem}[section]
    \newtheorem{lemma}[thm]{Lemma}

    \newtheorem{example}[thm]{Example}

\theoremstyle{definition}
    \newtheorem{definition}[thm]{Definition}
    \newtheorem{rem}[thm]{Remark}

\numberwithin{equation}{section}

\newcommand{\caseI}[1]{\medskip \noindent \textbf{Case #1:}}

\DeclarePairedDelimiter\ceil{\lceil}{\rceil}
\DeclarePairedDelimiter\floor{\lfloor}{\rfloor}
\DeclarePairedDelimiter\abs{\lvert}{\rvert}
\DeclarePairedDelimiter\absleft{\langle}{\rangle_{\textnormal{left}}}
\DeclarePairedDelimiter\absright{\langle}{\rangle_{\textnormal{right}}}
\DeclarePairedDelimiter\absside{\langle}{\rangle_*}
\DeclarePairedDelimiter\abssideother{\langle}{\rangle_{**}}

\newcommand{\Z}{\mathbb{Z}}

\newcommand{\R}{\mathbb{R}}

\newcommand{\T}{\mathbb{T}}

\renewcommand{\phi}{\varphi}
\newcommand{\aw}{\mathbf{a}}

\newcommand{\fw}{\mathbf{f}}
\newcommand{\eps}{\varepsilon}
\newcommand{\one}{\mathbbm{1}}

\newcommand{\fract}[1]{\left\{ #1 \right\}}
\DeclarePairedDelimiter\dist{\lVert}{\rVert}

\newcommand{\nsmall}{n_{\textnormal{small}}}
\newcommand{\nlarge}{n_{\textnormal{large}}}
\newcommand{\fleft}{f_{\textnormal{left}}}
\newcommand{\fright}{f_{\textnormal{right}}}
\newcommand{\xleft}{x_{\textnormal{left}}}
\newcommand{\xright}{x_{\textnormal{right}}}
\newcommand{\f}{f_{*}}
\newcommand{\x}{x_{*}}
\newcommand{\xstar}{\widehat{x}_*}
\newcommand{\xleftstar}{\widehat{x}_{\textnormal{left}}}
\newcommand{\xrightstar}{\widehat{x}_{\textnormal{right}}}

\begin{document}
\title[Balanced Rectangles]{Balanced rectangles over Sturmian words and
minimal discrepancy intervals}
\thanks{Research funded by the Austrian Science Fund (FWF) 10.55776/J4850.
}

\author[I. Vukusic]{Ingrid Vukusic}
\address{I. Vukusic,
Department of Mathematics, University of York, Ian Wand
Building, Deramore Lane, York, YO10 5GH, United Kingdom}
\email{ingrid.vukusic\char'100york.ac.uk}

\begin{abstract}
We consider $m\times n$ rectangular matrices formed from Sturmian words with slope $\alpha$, and we fully characterise their balance properties in terms of the Ostrowski representations of $m$ and $n$ with respect to $\alpha$.
This generalises recent results by Anselmo et al., as well as those by Shallit and the author, where only quadratic irrational slopes were considered.
In contrast to the two mentioned papers, the approach in this paper is based on the distribution of $n\alpha \bmod 1$. 
\end{abstract}

\maketitle

\section{Introduction}

Let $\alpha \in (0,1)$ be irrational and $\theta \in [0,1)$. Then the \emph{Sturmian word $\aw = a_1 a_2 a_3 \cdots$ with slope $\alpha$ and intercept $\theta$} can be defined via
\begin{equation}\label{eq:def}
    a_n := \floor{(n+1)\alpha + \theta} - \floor{n \alpha + \theta} \in \{0,1\}.
\end{equation}
For example, setting $\alpha = (3 - \sqrt{5})/2$ and $\theta = 0$ we get the famous infinite Fibonacci word 
\[
    \fw = 0 1 0 01 010 0100101\cdots.
\]
Recall that a \emph{factor} is simply a contiguous block of symbols within a word, and the \emph{weight} of a factor of a binary word is the number of $1$'s contained in it.
One of the basic properties of Sturmian words is that they are \emph{balanced}, that is, the weights of any two factors of the same length differ by at most $1$. (For example, in the Fibonacci word above, all factors of length $5$ either contain exactly one or exactly two $1$'s.)
In fact, the converse holds as well: every non-periodic balanced word is a Sturmian word.

The notion of balance was extended to multidimensional words by Berth\'e and Tijdeman~\cite{BertheTijdeman2002}. In particular, they proved that in dimension larger than $1$ only periodic words can be balanced.
Among other results, they also considered $2$-dimensional Sturmian words and gave a quantitative measure of their non-balancedness. 
Recently, Anselmo et al.\ \cite{AnselmoGiammarresiMadoniaSelmi2025} proved that a certain family of rectangles of the Fibonacci word is balanced, and a full characterisation was provided in \cite{ShallitVukusic2025}.
Let us discuss these results in a bit more detail because the goal of the present paper is to generalise them.

For every infinite word $\aw = a_1 a_2 a_3 \cdots$ let us define the infinite Hankel matrix $A = (a_{k, \ell})_{k\geq 1, \ell\geq 0}$ by $a_{k,\ell} := a_{k+\ell}$.
Then we can consider $m \times n$ submatrices
where the entry in the upper left corner has index sum $i$:
\[
    A(i,m,n)
    := \begin{pmatrix}
        a_i & a_{i+1} & \dots & a_{i+n-1} \\
        a_{i+1} & a_{i+2} & \dots & a_{i+n} \\
        \vdots & \vdots &     & \vdots \\
        a_{i+m-1} & a_{i+m} & \dots & a_{i+m+n-2} \\
    \end{pmatrix}.
\]
The sum over all entries in the matrix $A(i,m,n)$ equals
\begin{equation}\label{eq:T}
    T(i,m,n)
    := \sum_{k = 0}^{m-1} \sum_{\ell = 0}^{n-1} a_{i+k+\ell}.
\end{equation}
    
\begin{definition}\label{def:bal_rect}
Let $\aw = a_1 a_2 a_3\cdots$ be an infinite word over $\{0,1\}$. We say that the $m\times n$ rectangles of $\aw$ are \emph{balanced} if there exists an integer $c=c(\aw, m,n)$ such that
\[
    T(i,m,n) \in \{c, c+1\}
\]
for all $i\geq 1$.\footnote{This notion of balance for rectangles appears in \cite{Puzynina2019}, described via ``abelian complexity''.}
\end{definition}
For example, for the Fibonacci word it turns out that the $2\times 3$ rectangles are balanced (they always have weight $2$ or $3$), whereas the $2\times 4$ rectangles are not balanced (they can have weight $2$, $3$, or $4$).
Anselmo et al.~\cite{AnselmoGiammarresiMadoniaSelmi2025} proved that if $\max(m, n)$ is a Fibonacci number, then the $m\times n$ rectangles of the Fibonacci word are balanced.
In \cite{ShallitVukusic2025}, a full characterisation of balancedness was given in terms of the Zeckendorf representations of $m,n$. (The Zeckendorf representation of a positive integer is its unique representation as the sum of distinct and non-consecutive Fibonacci numbers).
Moreover, it was described how the software \emph{Walnut} can be used to do the same for every fixed quadratic irrational $\alpha$ and the corresponding representations. Note that the assumption that $\alpha$ is a quadratic irrational is essential for using Walnut, since quadratic irrationals are precisely the numbers with eventually periodic continued fraction expansion.

In this paper, we completely solve the $2$-dimensional balance problem for all irrationals~$\alpha$. 
The characterisation (Theorem~\ref{thm:bal_char_sturmian} in the next section) is in terms of the Ostrowski representations of $m,n$ with respect to $\alpha$.
The proof is based on Diophantine approximation and ideas used by Berthé and Tijdeman~\cite{BertheTijdeman2002}.
Berthé and Tijdeman also mentioned the connection between balance and so-called bounded remainder sets, which is a concept from dynamical systems/discrepancy theory.
We provide a little bit of background, as our main result will turn out to be equivalent to a specific statement about the distribution of $n\alpha \bmod 1$.
For some quick intuition on this, note that ~\eqref{eq:def} can equivalently be phrased as 
\[
    a_n = 1 :\iff \fract{n\alpha} \in [1-\alpha - \theta, 1- \theta).
\]
Here $\fract{x} = x - \floor{x}$ denotes the fractional part of $x$, and the interval is understood modulo $1$, i.e., in case $1-\alpha-\theta < 0$, we mean the interval that is ``wrapped around $0$'', $[\fract{1-\alpha - \theta}, 1) \cup [0, 1- \theta)$.
Thus, the exact distribution of $n\alpha \bmod 1$ contains full information on our Sturmian sequence.

It is well known (see, e.g., \cite[Chapt.~1]{Bugeaud_DistrModOne}) that for irrational $\alpha$ the sequence $(n\alpha)_{n\geq 0}$ is \emph{uniformly distributed modulo $1$}. 
In other words, if we consider the fractional parts $\fract{n\alpha}$ for $n = 0,1,2,\ldots$, every interval $I\subseteq [0,1)$ gets ``its fair share of points'' in the following sense:
\[
    \lim_{N\to\infty} \frac{\# \{ n \colon 0\leq n \leq N-1 \text{ and } \fract{n\alpha} \in I \} }{N}
    = |I|,
\]
where $|I|$ denotes the length of $I$.
Of course, not all intervals can get ``exactly their fair share of points'' if we consider finite sets of points. This is quantified by the \emph{discrepancy}
\[  
    D_N((n\alpha)_{n\geq 0}) = \sup_{I \subseteq[0,1]} \left|
    \frac{\# \{ n \colon 0\leq n \leq N-1 \text{ and } \fract{n\alpha} \in I \} }{N}
    - |I|
    \right|.
\]
In view of this, it seems appropriate to say that an interval $I$ has ``minimal discrepancy with respect to $\alpha$ and $N$'', if
\begin{equation}\label{eq:discrep_one}
    \big|
    \# \{ n \colon 0\leq n \leq N-1 \text{ and } \fract{n\alpha} \in I \}
    - N|I|
    \big|
    < 1.
\end{equation}
This again is somewhat related to bounded remainder sets, where the bound $1$ is relaxed to $C$ but has to be satisfied for all $N$.

In this paper, we are interested in when all intervals of a fixed length have minimal discrepancy for a fixed $N$.
More specifically, the balancedness of the $m\times n$ rectangles will turn out to be equivalent (see Theorem~\ref{thm:def-equiv}) to the balancedness of intervals of length $\fract{n\alpha}$ with respect to $\alpha$ and $m$, in the following sense.

\begin{definition}\label{def:bal_int}
Let $\alpha, \delta \in (0,1)$ and let $N\geq 1$ be an integer. We say the intervals of length $\delta$ are \emph{balanced with respect to $(\alpha, N)$} if
there exists an integer $c = c(\alpha, \delta, N)$ such that
for all half open
 intervals $I = [\xi, \xi + \delta)$, $0\leq \xi < 1$, we have 
\begin{equation}\label{eq:def_bal_intervals}
    \# \{ n \colon 0\leq n \leq N-1 \text{ and } \fract{n\alpha} \in I \} 
    \in \{c, c+1\}.
\end{equation}
    
Note that the intervals $I$ are understood modulo $1$, i.e., if $\xi+\delta \geq 1$ then the interval is $I = [\xi,1) \cup [0,\fract{\xi+\delta})$.
\end{definition}

It is not hard to see that \eqref{eq:def_bal_intervals} is indeed closely related to \eqref{eq:discrep_one}, justifying the second part of the title of this paper.

In the next section, we state the full characterisation of balanced $m \times n$ rectangles.
In Section~\ref{sec:equiv} we prove the equivalence between balanced rectangles and balanced intervals.
Then we make preparations for proving the characterisation of balanced intervals of length $\fract{n\alpha}$: In Section~\ref{sec:bij} we rephrase balancedness in terms of bijectivity of a certain function; in Section~\ref{sec:Ostr} we recall some specific properties of Ostrowski representations.
Finally, in Section~\ref{sec:proof}, we prove the full characterisation of balanced $m\times n$ rectangles.

We conclude this introduction with two remarks on the parameters $\alpha$ and $\theta$ of Sturmian words.

\begin{rem}\label{rem:theta}
In the rest of the paper, we will assume $\theta = 0$. This can be justified by the following:
It is a basic fact (see, e.g., \cite[Theorem 10.5.3]{AlloucheShallit_Automatic}) that two Sturmian words have the same slope if and only if they have the same sets of factors.
Since each rectangle $A(i,m,n)$ is fully determined by the factor $a_i \cdots a_{i+m+n-2}$, the infinite Hankel matrices corresponding to two Sturmian words with the same slope have exactly the same rectangles.
\end{rem}

\begin{rem}\label{rem:alpha}
We will also often assume $\alpha < 1/2$.
It is easy to check that the Sturmian word with slope $1-\alpha$ and intercept $0$ can be obtained by flipping the digits of the Sturmian word with slope $\alpha$ and intercept $0$.
Of course, for every pair $(m,n)$, the $m\times n$ rectangles of one sequence are balanced if and only if the $m\times n$ rectangles of the other sequence are.
Therefore, if $\alpha > 1/2$, we can equivalently consider $1-\alpha < 1/2$ instead.
\end{rem}

\section{Full characterisation via Ostrowski representations}\label{sec:mainresult}

In this section we state our main result, namely the full characterisation of the $m\times n$ balanced rectangles. But first, let us briefly recall continued fractions and the Ostrowski representation; see, e.g., \cite{RockettSzusz_contfrac} for a reference book.

Every irrational real number $\alpha$ can be uniquely represented by its infinite \emph{simple continued fraction expansion}
\[
	\alpha 
	= [a_0; a_1, a_2, \ldots]
	= a_0 + \frac{1}{a_1 + \frac{1}{a_2 + \frac{1}{\ddots}}},
\]
where $a_0$ is an integer and $a_1, a_2, \ldots$ are positive integers, called \emph{partial quotients}.
We can truncate the continued fraction expansion of $\alpha$ at its $k$-th partial quotient, and obtain the rational number $p_k/q_k = [a_0; a_1,\ldots ,a_k]$, 
called the $k$-th \emph{convergent} to $\alpha$.
These convergents are famously particularly good approximations to $\alpha$. 
In particular, the numbers $q_k \alpha$ are very close to an integer (at least for large $k$), or, in other words, very close to $0$ modulo $1$. 
Since we are interested in $n\alpha \bmod 1$, the natural way to represent an integer $n$ therefore is to write it as the sum of $q_k$'s, using a greedy algorithm. This is known as the Ostrowski representation. To be precise, every positive integer $n$ has a unique representation 
\[
    n = \sum_{k = 0}^N b_k q_k
\]
with $b_N \neq 0$, $0 \leq b_k \leq a_{k+1}$ for $k \geq 1$ and $0 \leq b_0 \leq a_1 - 1$, and the additional rule that $b_{k-1} = 0$ whenever $b_{k} = a_{k+1}$. 
For more properties of the Ostrowski representation, see Section~\ref{sec:Ostr}.
In the rest of the paper, if we write an expression of the shape $n = \sum_{k = L}^N b_k q_k$, we always imply that it is a valid Ostrowski representation with respect to $\alpha$, but we do not necessarily assume $b_L,b_N>0$.

Now we state our characterisation of balanced $m\times n$ rectangles of a Sturmian word with slope $\alpha$ in terms of the Ostrowski representations of $m,n$.
Since the infinite Hankel matrix $A = (a_{k, \ell})_{k\geq 1, \ell\geq 0}$ is symmetric, the balance problem for rectangles is symmetric, and from now on we assume $m\leq n$.
Moreover, we can assume $m\geq 2$ because for $m=1$ the rectangles are just the factors of the $1$-dimensional Sturmian word, which are of course balanced.
It turns out that there are essentially  only two situations when the rectangles with $2\leq m \leq n$ are balanced:
\begin{itemize}
\item The integer $m$ only has small digits, and $n$ only has large digits in its Ostrowski representation with respect to $\alpha$. In the edge case, where $m$ and $n$ share exactly one digit, there are some extra conditions.
\item The integer $m$ is either the denominator of a convergent or of a semi-convergent. 
(A semi-convergent is a number of the shape $(p_{k-1} + a p_{k})/(q_{k-1} + a q_{k})$ with $1 \leq a \leq a_{k+1}-1$.)
Moreover, parity restrictions on the corresponding digits in $n$ apply.
\end{itemize}

\begin{thm}\label{thm:bal_char_sturmian}
Let $\alpha\in (0,1)$ be irrational and $2 \leq m \leq n$. Then the $m\times n$ rectangles of the Sturmian words with slope $\alpha$ are balanced if and only if the Ostrowski representations of $m,n$ with respect to $\alpha$ are of at least one of the following four shapes.

They have ``split representations'' in the following sense:
\begin{enumerate}[label = (\roman*)]
    \item\label{it:orig_split1} 
    $m = \sum_{k = 0}^M b_k q_k$ and 
    $n = \sum_{k = M+1}^N b_k q_k$;
    \item\label{it:orig_split2} 
    $m = \sum_{k = 0}^M b_k q_k$ with $b_M \neq 0$, and 
    $n = q_M + \sum_{k = M + 1 + 2t}^N b_k q_k$ with $t\geq 0$ and $b_{M+1+2t} \neq 0$.
\end{enumerate}

The smaller number $m$ is the denominator of a (semi-)convergent, and we have certain parity restrictions on the large digits in $n$:
\begin{enumerate}[label = (\roman*)]\setcounter{enumi}{2}
    \item\label{it:orig_conv} 
    $m = q_M$ and $n = \sum_{k = 0}^{M-1} b_k q_k + \sum_{k = M+2t}^{N} b_k q_k$ with $t\geq 0$ and $b_{M+2t}\neq 0$;
    \item\label{it:orig_semiconv} 
    $m = q_{M-1} + a q_{M} $ with $1 \leq a \leq a_{M+1} - 1$ and 
    $n = \sum_{k = 0}^{M-1} b_k q_k + \sum_{k = M+1+2t}^{N} b_k q_k$ with $t\geq 0$ and $b_{M+1+2t}\neq 0$.
\end{enumerate}
\end{thm}

\begin{example}
For $\alpha = \pi/4$ we have $\alpha = [0; 1, 3, 1, 1, 1, 15, 2, 72, \ldots]$, and $q_0 = 1$, $q_1 = 1$, $q_2 = 4$, $q_3 = 5$, $q_4 = 9$, \ldots.
Thus, if we want to know whether the $2\times 10$ rectangles are balanced, we can use a greedy algorithm to get the Ostrowski representations
\begin{align*}
    m &= 2 = 2q_1, \\
    n &= 10 = q_4 + q_1.
\end{align*}
This falls under the case \ref{it:orig_split2} in Theorem~\ref{thm:bal_char_sturmian}, and so we know that the $2\times 10$ rectangles are balanced.
Indeed, one can check that they always have weight $15$ or $16$.

On the other hand, let us consider $m = 5$, $n=10$:
\begin{align*}
    m &= 5 = q_3, \\
    n &= 10 = q_4 + q_1.
\end{align*}
This does not correspond to any of the cases in Theorem~\ref{thm:bal_char_sturmian}, and indeed one can check that the $5\times 10$ rectangles can have weight $38$, $39$, or $40$.
\end{example}

\begin{rem}\label{rem:Ostr-flip-alpha}
In Remark~\ref{rem:alpha} we mentioned that from a balance point of view it doesn't matter if we consider $\alpha$ or $1-\alpha$.
Indeed, the same is true for the shape of the Ostrowski representations.
The next lemma implies, in particular, that the Ostrowski representations of $m,n$ with respect to $\alpha$ are of one of the four special shapes in Theorem~\ref{thm:bal_char_sturmian} if and only if those with respect to $1-\alpha$ are.
Together with Remark~\ref{rem:alpha}, this allows us to assume $\alpha < 1/2$ in the rest of the paper, without loss of generality.  
\end{rem}

Note that for $\alpha = [0; a_1, a_2, \ldots]$ we have $\alpha > 1/2 \iff a_1 = 1$, in which case, by the rules of the Ostrowski representation, we have $b_0 = 0$. In other words, if $\alpha >1/2$, then the first possible nonzero term in the representation is $b_1 q_1$. 

\begin{lemma}\label{lem:1-alpha-Ostr}
Let $\alpha < 1/2$ be irrational and $n$ a positive integer.
Then the Ostrowski representation of $n$ with respect to $\alpha$ is $n = \sum_{k = 0}^N b_k q_k$ if and only if the Ostrowski representation of $n$ with respect to $1-\alpha$ is  $n = \sum_{k = 1}^{N+1} b_{k-1} q_k$.
\end{lemma}
\begin{proof}
For $\alpha < 1/2$ and $\alpha = [0; a_1, a_2, a_3, \ldots]$ one can check that $1-\alpha = [0;1,a_1-1,a_2, a_3,\ldots]$.
Then from the recurrence formula for convergents (see also \eqref{eq:conv_recursion} in Section~\ref{sec:convergents}) one can see that $\alpha$ and $1-\alpha$ have the same sequence of denominators $q_0,q_1,q_2,\ldots$, except for the index shift.
\end{proof}

\section{Correspondence between balancedness and low discrepancy intervals}
\label{sec:equiv}

In this section we prove the equivalence between balancedness of rectangles and balancedness of intervals.
We start by checking the well known fact that Sturmian words are balanced, as the corresponding formula will be useful in a moment. 
We can directly compute the weight of a factor of length $n$ starting at index $i$ by using \eqref{eq:def} with $\theta = 0$ and telescoping:
\begin{equation}\label{eq:blocksum}
\begin{split}
    a_i + a_{i+1} + \dots + a_{i+n-1}
    &= \floor{(i+n)\alpha} - \floor{i\alpha}\\
    &= \floor{n\alpha} + \begin{cases}
        1, & \text{if } \fract{n\alpha} + \fract{i\alpha} \geq 1;\\
        0, & \text{else}.
    \end{cases}
\end{split}
\end{equation}
Therefore, the weight of every factor of length $n$ is either $\floor{n\alpha}$ or $\floor{n\alpha} + 1$, and in particular the factors of length $n$ are balanced for every $n$.

We can now use this to compute $T(i,m,n)$ (defined in \eqref{eq:T}) by summing over the rows of $A(i,m,n)$.
Note that the conditional expression in \eqref{eq:blocksum} can be expressed using the indicator function in the following way:
\[
    \one_{[1-\fract{n\alpha}, 1)}(\fract{i\alpha}).
\]
Thus, we obtain 
\begin{align*}
    T(i,m,n)
    &= \sum_{\ell = 0}^{m-1} (a_{i+\ell} + a_{i+\ell+1} + \dots + a_{i+\ell+n-1})\\
    &= m \floor{n\alpha} + \sum_{\ell = 0}^{m-1} 
        \one_{[1-\fract{n\alpha}, 1)}(\fract{(i+\ell)\alpha}).
\end{align*}
Of course, the value of $m \floor{n\alpha}$ is independent of $i$, and so the $m \times n$ rectangles are balanced if and only if the sum
\begin{align*}\label{eq:S}
    S(i,m,n) 
    := \sum_{\ell = 0}^{m-1} 
        \one_{[1-\fract{n\alpha}, 1)}(\fract{(i+\ell)\alpha})
\end{align*}
takes exactly two values for all $i$.
This and the fact that $i\alpha \bmod 1$ is dense in $[0,1)$ lead to the following theorem. (Recall Definition~\ref{def:bal_int} for the balancedness of intervals.)

\begin{thm}\label{thm:def-equiv}
Let $\aw$ be a Sturmian word with slope $\alpha$. Then the $m\times n$ rectangles are balanced if and only if the intervals of length $\fract{n\alpha}$ are balanced with respect to $(\alpha, m)$.
\end{thm}
\begin{proof}
As explained above, the $m\times n$ rectangles are balanced if and only if the sum $S(i,m,n)$ is balanced as a sequence indexed by $i$.
Note that we can rewrite $S(i,m,n)$ as
\[
    S(i,m,n) 
	= \sum_{\ell = 0}^{m-1} 
        \one_{[1-\fract{n\alpha}-\fract{i \alpha}, 1-\fract{i \alpha})}(\fract{\ell\alpha}).
\]
In other words, the $m\times n$ rectangles are balanced if and only if for every $i\geq 1$ the interval $[1-\fract{n\alpha}-\fract{i\alpha}, 1-\fract{i\alpha})$ contains either $c$ or $c+1$ of the points $\fract{0},\fract{\alpha}, \fract{2 \alpha}, \ldots, \fract{(m-1)\alpha}$ for some fixed $c$. Here, and everywhere else, the intervals are understood modulo $1$.
Thus, the balancedness of the intervals of length $\fract{n\alpha}$ with respect to $(\alpha, m)$ clearly implies the balancedness of the $m\times n$ rectangles.

For the reversed implication let $[\xi, \xi+ \fract{n\alpha})$ be an arbitrary interval of length $\fract{n\alpha}$.
If $[\xi, \xi+ \fract{n\alpha})$ contains none of the points $\fract{\ell\alpha}$ with $0\leq \ell \leq m-1$, set $\eps_1 := 1$. 
Otherwise, set 
\[
    \eps_1 := \min_{0\leq \ell \leq m-1} \fract{\xi+\fract{n\alpha} - \fract{\ell\alpha}},
\]
i.e., $\eps_1$ is the distance between the right endpoint of the interval and the rightmost point $\fract{\ell\alpha}$ contained in the interval.
Similarly, set 
\[
    \eps_2 : = \min_{\substack{0\leq \ell \leq m-1 \\\fract{\ell\alpha} \neq \xi}} \fract{\xi - \fract{\ell\alpha}},
\]
i.e., $\eps_2$ is the distance between the left endpoint of the interval and the closest point  $\fract{\ell\alpha}$  lying strictly to the left of the interval.
Set $\eps := \min\{\eps_1, \eps_2\}$
and recall that $i\alpha \bmod 1$ is dense in $[0,1)$. 
Thus, there exists an $i\geq 1$ such that $1 - \fract{n\alpha} - \fract{i\alpha} \in (\xi - \eps, \xi]$.
Then, by construction, the interval $[\xi, \xi+ \fract{n\alpha})$ contains exactly the same points as the interval $[1 - \fract{n\alpha} - \fract{i\alpha}, 1 -\fract{i\alpha})$.
Therefore, the balancedness of the $m\times n$ rectangles implies the balancedness of the intervals of length $\fract{n\alpha}$.
\end{proof}

\begin{rem}
As mentioned before,
since the infinite Hankel matrix $A = (a_{k, \ell})_{k\geq 1, \ell\geq 0}$ is symmetric, the balance problem for rectangles is symmetric.
From Theorem~\ref{thm:def-equiv} we immediately get the following fact:
Let $m,n\geq 1$ be integers. Then the intervals of length $\fract{m\alpha}$ are balanced with respect to $(\alpha, n)$ if and only if the intervals of length $\fract{n\alpha}$ are balanced with respect to $(\alpha, m)$.
\end{rem}

\begin{rem}
In view of Theorem~\ref{thm:def-equiv}, some of the cases of Theorem~\ref{thm:bal_char_sturmian} become quite obvious, provided one is familiar with the basic properties of Ostrowski representations. For example, the cases \ref{it:orig_split1} and \ref{it:orig_split2}, where $n$ has only large digits, correspond to the intervals of length $\fract{n\alpha}$ being extremely short (or extremely long, but then one can think of the complements). In fact, it is not very hard to prove that in the cases \ref{it:orig_split1} and \ref{it:orig_split2} the corresponding intervals each contain at most one point, and thus they are balanced. 

For some of the other cases, one can use other known tricks as well. 
For example, if $m = q_M$, the points $\fract{\ell \alpha}$, for $0\leq \ell \leq m-1$, are particularly evenly distributed and one can use the trick that $\alpha \approx p_M/q_M$, and therefore $\fract{\ell \alpha} \approx (\ell p_M \bmod q_M)/q_M$, to characterise balancedness.

Overall, one can say that the intervals are balanced in the two following cases: Either the intervals are very short (or very long); this corresponds to the cases \ref{it:orig_split1} and \ref{it:orig_split2} in Theorem~\ref{thm:bal_char_sturmian}. Or the points are very evenly distributed, and the interval lengths are slightly longer (or slightly shorter) than the distance
between certain two points; this corresponds to the cases \ref{it:orig_conv} and \ref{it:orig_semiconv} in Theorem~\ref{thm:bal_char_sturmian}.

In this paper we want to deal with all cases in a somewhat uniform way. To that aim, we rephrase balancedness of intervals in terms of bijectivity of a certain map $\f$. This is done in the next section. Later, we will rephrase the bijectivity of $\f$ again in terms of another function, which is more closely related to approximation properties of Ostrowski representations.
\end{rem}

\section{Balanced intervals and certain bijective maps}\label{sec:bij}

We start by defining balancedness of intervals in a slightly more general setting because this makes the arguments clearer.
Note that, as before, all intervals are understood modulo $1$. To make this more rigorous, we speak of the torus $\T = \R/\Z$, which can be thought of as the interval $[0,1)$, where we compute modulo $1$.

\begin{definition}
Let $B = \{\xi_0, \xi_1, \ldots, \xi_{m-1}\}$ be a set of $m$ distinct points on the torus $\T$ and let $\delta \in (0,1)$. We say that the intervals of length $\delta$ are \emph{balanced with respect to $B$} if there exists an integer $c = c(B,\delta)$ such that for all $\xi\in \T$ we have
\[
    \# ( [\xi,\xi+\delta) \cap B ) \in \{c, c+1\}.
\]

\end{definition}

The goal of this section is to find a way to determine this balance property without actually counting the points in each interval. This morally corresponds to the idea in \cite[Lemma~1]{ShallitVukusic2025}. 
The first step is to focus on the intervals
 $[\xi_\ell, \xi_\ell+\delta)$ for $0 \leq \ell \leq m-1$ and to find the closest points $\xi_j \in B$ to the left and to the right of $\xi_\ell+\delta$. 
We define the corresponding functions in terms of the indices of the points.

\begin{definition}\label{def:flr}
Let  $B = \{\xi_0, \xi_1, \ldots, \xi_{m-1}\} \subseteq \T$ be a set of $m$ distinct points
and let $\delta \in (0,1)$.
Then we define the following two maps on the set $\{0,1,\ldots, m-1\}$: $\fleft$ maps $\ell$ to the index of the closest point in $B$ that lies to the left of $\xi_\ell + \delta$, and $\fright$ maps $\ell$ to the index of the closest point in $B$ that lies to the right of $\xi_\ell + \delta$. In other words, 
\begin{align*}
    \fleft(\ell) = j
    \quad &:\iff \quad
    \xi_\ell + \delta - \xi_j = \min_{0 \leq i \leq m-1} (\xi_\ell + \delta) - \xi_i, \\
    \fright(\ell) = j
    \quad &:\iff \quad
     \xi_j - (\xi_\ell + \delta) = \min_{0 \leq i \leq m-1} \xi_i - (\xi_\ell + \delta),
\end{align*}
where everything is taken modulo $1$ and then ordered in the usual way in $[0,1)$.
\end{definition}

Now in the case that the interval length $\delta$ does not match the distance between any two points in $B$, we have the following main lemma.

\begin{lemma}\label{lem:bal_fg}
Let  $B = \{\xi_0, \xi_1, \ldots, \xi_{m-1}\} \subseteq \T$ be a set of $m$ distinct points.
Moreover, let $\delta \in (0,1)$ be such that $\delta \neq \xi_i - \xi_j$ for all $i,j$.
Then the following statements are equivalent:
\begin{enumerate}[label = (\alph*)]
    \item\label{it:bal_int} The intervals of length $\delta$ are balanced with respect to $B$.
    \item\label{it:f_left} The function $\fleft$ is bijective on $\{0, 1, \ldots, m-1\}$.
    \item\label{it:f_right} The function $\fright$ is bijective on $\{0, 1, \ldots, m-1\}$.
\end{enumerate}
\end{lemma}
\begin{proof}
Assume without loss of generality that $0 \leq \xi_0 < \xi_1 < \dots < \xi_{m-1} <1$. Moreover, in the following we use the notation $\xi_{m} := \xi_0$.

Start with the interval $[0,\delta)$ and observe what happens as we shift it continuously to the right, i.e., consider $[\xi, \xi+\delta)\subset \T$ as $\xi$ runs through $[0,1]$.
Every time $\xi$ increases from $\xi_\ell$ to $\xi_\ell + \eps$ for some $\ell$ and small $\eps>0$, we ``lose'' the point $\xi_\ell$.
Every time $\xi$ increases from $\xi_j - \delta$ to $\xi_j - \delta + \eps$ for some $j$, we ``gain'' the point $\xi_j$.
Since $\delta \neq \xi_j - \xi_\ell$ for all $\ell,j$, we never gain a point and lose a point at exactly the same time.
Therefore, the intervals of length $\delta$ are balanced if an only if, as we slide our interval across the torus, we always alternate between gaining and losing a point. 
In other words, the intervals of length $\delta$ are balanced if an only if for every $\ell$ there exists a unique $j$ such that $\xi_j$ lies between $\xi_\ell +\delta$ and $\xi_{\ell+1} + \delta$. (This corresponds to the fact that as we shift from $[\xi_\ell,\xi_\ell+\delta)$ to $[\xi_{\ell+1}, \xi_{\ell+1} + \delta)$ we lose $\xi_\ell$ and then gain $\xi_j$, before losing $\xi_{\ell+1}$.)
Moreover, note that in this situation $j = \fright(\xi_\ell) = \fleft(\xi_{\ell+1})$, and thus balancedness implies that $\fleft,\fright$ are bijective.

For the other implication, note that $\fleft$ being bijective or $\fright$ being bijective each imply that for every $\ell$ there exists a unique $j$ such that $\xi_j$ lies between $\xi_\ell +\delta$ and $\xi_{\ell+1} + \delta$. As described above, this is equivalent to the intervals of length $\delta$ being balanced.
\end{proof}

\begin{rem}
The assumption in Lemma~\ref{lem:bal_fg} that $\delta \neq \xi_i - \xi_j$ for all $i,j$ is necessary. For example, consider the three points $\xi_0 = 0$, $\xi_1 = 1/4$, and $\xi_2 = 1/2$, and set $\delta = 1/2$. Then it is easy to check that the half open intervals of length $\delta$ contain either $1$ or $2$ points, and thus are balanced.
However,
$\fleft(0) = 2 = \fleft(1)$ and $\fright(2) = 0 = \fright(1)$.
One can also check that changing the definitions of $\fleft,\fleft$ to contain the restriction ``strictly to the left/right'' does not resolve the issue.
\end{rem}

\begin{rem}\label{rem:Ad_appl}
In our application of Lemma~\ref{lem:bal_fg}
in Section~\ref{sec:proof} we will have 
\[
	B = \{0, \alpha, \fract{2\alpha}, \ldots, \fract{(m-1)\alpha} \} 
\]
and 
\[
	\delta = \fract{n\alpha}.
\]
Then, if we assume $n\geq m$, we indeed get that $\delta\neq \fract{i\alpha} - \fract{j\alpha}$ for all $0\leq i,j \leq m-1 <n$.
\end{rem}

Finally, we state a simple lemma which will also be useful later.
\begin{lemma}\label{lem:alternative_intervals}
Let  $B = \{\xi_0, \xi_1, \ldots, \xi_{m-1}\} \subseteq \T$ be a set of $m$ distinct points and let $\delta \in (0,1)$.
Then the intervals of length $\delta$ are balanced with respect to $B$ if and only if the intervals of length $1-\delta$ are balanced.
Moreover, if $\delta \neq \xi_i - \xi_j$ for all $i,j$, then
for all $\eps$ with $\abs{\eps}$ sufficiently small, the intervals of length $\delta$ are balanced if and only if the intervals of length $1-\delta+\eps$ are balanced.
\end{lemma}
\begin{proof}
The first statement is clear because the half open intervals of length $1-\delta$ correspond to the complements of the intervals of length $\delta$.
If $\delta \neq \xi_i - \xi_j$ for all $i,j$, then it is clear that for $\abs{\eps} < \min_{i,j} \abs{\delta - (\xi_i - \xi_j)}$ we can modify the interval length $\delta$ to $\delta+\eps$ without changing the occurring values of $\# ([\xi,\xi+\delta) \cap B)$.
\end{proof}

\section{Properties of convergents and Ostrowski representations}\label{sec:Ostr}

The strategy for proving our main result (Theorem~\ref{thm:bal_char_sturmian}) in the next section will be to apply
Lemma~\ref{lem:bal_fg} to the points $\xi_\ell = \fract{\ell\alpha}$ with $0 \leq \ell \leq m-1$ and $\delta = \fract{n\alpha}$.
Therefore, the right endpoints of the intervals $[\xi_\ell, \xi_\ell +\delta)$ will be of the shape $\fract{\ell\alpha} + \fract{n\alpha} = \fract{(\ell+n)\alpha}$, where $\ell+n > m-1$.
In order to determine $\fleft(\ell), \fright(\ell)$ we will need to figure out which point $\fract{j\alpha}$ with $0\leq j \leq m-1$ lies closest to the left and which one closest to the right of $\fract{(\ell+n)\alpha}$.
We will do this using the Ostrowski representations mentioned in Section~\ref{sec:mainresult}.

We now recall several properties of Ostrowski representations which that be useful. 
All lemmas in this section are probably well known to experts (except perhaps for Lemma~\ref{lem:flip_parity}, but this one is not very hard to see either).
Rather unfortunately though, there does not seem to exist a suitable reference book. In order to provide proofs for the lemmas, we recall some basic properties of convergents first; see, e.g., \cite{RockettSzusz_contfrac} for a reference. 
Throughout the rest of this paper, $\alpha \in (0,1)$ is a fixed irrational, and all Ostrowski representations are with respect to $\alpha$.
In fact, in view of Lemma~\ref{lem:1-alpha-Ostr} and Remark~\ref{rem:Ostr-flip-alpha}, we assume $\alpha < 1/2$.

\subsection{Basic properties of convergents}\label{sec:convergents}

As before, let $p_k/q_k$ denote the convergents to $\alpha\in(0,1/2)$.
Then the numerators and denominators follow the recursions 
\begin{align}
    p_0 = 0, \quad
    p_1 = 1, \quad
    p_{k+1} = a_{k+1}p_k +  p_{k-1}
    \quad \text{for } k\geq 1; \nonumber \\
    q_0 = 1, \quad
    q_1 = a_1, \quad
    q_{k+1} = a_{k+1}q_k +  q_{k-1}
    \quad \text{for } k\geq 1.\label{eq:conv_recursion}
\end{align}
Convergents are the \emph{best approximations} in the following sense:
Let $\dist{\xi} := \min_{x\in\Z} \abs{\xi - x}$ denote the distance to the nearest integer to $\xi$.
Then we have
\begin{equation}\label{eq:best_approx_prop}
    0 <  q < q_{k+1} 
    \implies
    \dist{q\alpha} \geq \dist{q_k \alpha}.
\end{equation}
We set
\[
    \delta_k : = q_k \alpha - p_k.
\]
Then we have (see, e.g., \cite[p.~9]{RockettSzusz_contfrac})
\begin{equation}\label{eq:delta_sign}
    \delta_k = (-1)^k \dist{q_k \alpha}
    \quad \text{for } k\geq 0,
\end{equation}
and, moreover, 
\begin{equation}\label{eq:qkalpha_easybound}
	\abs{\delta_k}    
    = \dist{q_k \alpha} 
    \leq 1/q_{k+1}. 
\end{equation}
In particular, since $\dist{q_k\alpha}\leq 1/q_{k+1} \leq 1/2$ for $k\geq 0$, we have
\begin{equation}\label{eq:qk_parity}
    \fract{q_k \alpha} = \begin{cases}
        \dist{q_k \alpha}, &\text{if $k$ is even};\\
        1- \dist{q_k \alpha}, &\text{if $k$ is odd}.
    \end{cases}
\end{equation}

\subsection{Basic properties of Ostrowski representations}\label{sec:Ostr-details}

Recall that the Ostrowski representation of $n$ with respect to $\alpha$ is unique and of the shape
$n = \sum_{k = 0}^N b_k q_k$,
with $0 \leq b_k \leq a_{k+1}$ for $k \geq 1$, and $0 \leq b_0 \leq a_1 - 1$, and the additional rule that $b_{k-1} = 0$ whenever $b_{k} = a_{k+1}$.
The Ostrowski representation of $n$ can be obtained by a greedy algorithm. In particular, we have
\begin{equation}\label{eq:greedy}
    n = \sum_{k = 0}^N b_k q_k 
    \quad \text{with } b_N \neq 0
    \iff q_N \leq n < q_{N+1}.
\end{equation}

The Ostrowski representation is ``the correct way to represent integers'' in the following sense:
Roughly speaking, the digits of $n$ with small index determine where $\fract{n\alpha}$ lies in $[0,1)$, and the digits with large index produce a small error term.
This is plausible in view of \eqref{eq:qkalpha_easybound}.
Let us be more precise.

First, note that from \eqref{eq:conv_recursion} and \eqref{eq:delta_sign} 
it follows that
\begin{equation}\label{eq:sum-everyother-full-convergents}
    \sum_{t = 0}^\infty a_{L+2+2t} \dist{q_{L+2t+1}\alpha}
    = \dist{q_L\alpha},
\end{equation}
for $L\geq 0$ (recall that we are assuming $\alpha<1/2$).

Let $b_k(n)$ denote the coefficient of $q_k$ in the Ostrowski representation of $n$,
and let $k_0(n)$ be the index of the first nonzero coefficient in the Ostrowski representation of $n$, i.e.,
\[
    k_0(n) : = \min\{k \colon b_k(n) > 0\}.
\]
The next lemma describes the fact that if $k_0(n)$ is large, then $n\alpha$ is close to $0$ modulo $1$.
In particular, this shows that indeed digits with large index only have little impact on the position of $\fract{n\alpha}$.
Later in this paper, we will need somewhat precise bounds for the size of $\dist{n\alpha}$ if $n$ only has large digits, which is why the lemma is a bit lengthy. For quick intuition, focus on the second estimate.

\begin{lemma}\label{lem:Ost_largedigits}
Let $\alpha \in (0,1/2)$ and $n\geq 1$ with $k_0(n) = L$. Then
\begin{enumerate}
    \item\label{it:dist_lb} $\dist{n\alpha} > \dist{q_{L+1}\alpha}$;
    
    \item\label{it:dist_ub} $\dist{n\alpha} < \dist{q_{L-1}\alpha}$, provided that $L\geq 1$.
\end{enumerate}
Some other, more precise estimates are the following:
\begin{enumerate}\setcounter{enumi}{2}

    \item\label{it:dist_lb_N} $\dist{n\alpha} \geq \dist{q_{L+1}\alpha} + \dist{q_{N+2}}$ if $n < q_{N+1}$;

    \item\label{it:dist_lb2} $\dist{n\alpha} > \dist{q_{L}\alpha} + \dist{q_{L+1}\alpha}$ if $b_L(n) \geq 2$;
        
    \item\label{it:dist_ub_N} $\dist{n\alpha} < \dist{q_{L-1}\alpha} - \dist{q_{N+2}}$, provided that $L\geq 1$ and $n < q_{N+1}$;
    
    \item\label{it:dist_ub2} $\dist{n\alpha} < \dist{(q_{L-1} + q_L) \alpha}$ if $b_L(n) \leq a_{L+1} - 1$ and $L\geq 1$;

    \item\label{it:dist_compare} $\dist{n\alpha}<\dist{n'\alpha}$ if $k_0(n') = k_0(n) = L$ and $b_L(n)<b_L(n')$;
    
    \item\label{it:dist_specialsmall} $\dist{n\alpha} < \dist{q_L\alpha}$ if and only if 
    the Ostrowski representation of $n$ is of the shape 
    $n = q_L + \sum_{k = L+1 + 2t}^N b_k q_k$ with $t\geq 0$ and $b_{L+1+2t}>0$;

    \item\label{it:dist_notspecial} $\dist{n\alpha} \geq \dist{q_L\alpha} + \dist{q_N \alpha}$ if $q_L < n < q_{N+1}$ and $n$ is not of the shape $n = q_L + \sum_{k = L+1 + 2t}^N b_k q_k$ with $t\geq 0$ and $b_{L+1+2t}>0$.
    
\end{enumerate}
\end{lemma}
\begin{proof}
All statements follow pretty straightforwardly from  \eqref{eq:delta_sign}, \eqref{eq:sum-everyother-full-convergents}, and the rules for the digits of Ostrowski representations. 
See also, for example, \cite[p.~24--25, Lemma~1 and Theorem~1]{RockettSzusz_contfrac}.
\end{proof}

The next lemma follows from similar arguments; see, for example, \cite[Lemma~4.1]{BeresnevichHaynesVelani2020} for a reference.

\begin{lemma}\label{lem:Ostr_direction}
Let $\alpha \in (0,1/2)$ and $n\geq 1$ and assume that $k_0(n)\geq 1$.
Then we have
\[
    \dist{n\alpha}
    = \begin{cases}
        \fract{n\alpha}, &\text{if $k_0(n)$ is even};\\
        1 - \fract{n\alpha}, &\text{if $k_0(n)$ is odd}.
    \end{cases}
\]
\end{lemma}

Lemma~\ref{lem:Ostr_direction} cannot be extended to $k_0(n)\geq 0$, so when the smallest allowed digit in the representation of $n$ shows up, then $\fract{n\alpha}$ might lie on ``the wrong side of $1/2$''. Indeed, if $k_0(n)=0$, we might have $\fract{n\alpha} >1/2$ if $b_0(n)$ is large, even though from the parity of $k_0(n)$ we would expect $\fract{n\alpha}<1/2$.
The next lemma will be useful when dealing with such exceptional cases.

\begin{lemma}\label{lem:q0-wrongside-improve}
Assume $\alpha<1/2$ and $k_0(n) = 0$.
Then 
\[  
    \fract{n\alpha}>1/2 
    \quad \implies \quad
    \fract{n\alpha} < \fract{(n+1)\alpha}.
\]
Note that the implied statement is equivalent to $\dist{(n+1)\alpha} < \dist{n\alpha}$.
\end{lemma}
\begin{proof}
If $n = \sum_{k = 0}^N b_k q_k$ with $b_0>0$,
then we get from \eqref{eq:delta_sign} and \eqref{eq:sum-everyother-full-convergents} that
\begin{align*}
    \fract{n\alpha}
    &\leq \fract{b_0 q_0 \alpha} + \fract{(b_1 q_1 + \dots)\alpha}\\
    &< \fract{b_0 q_0 \alpha} + \fract{(a_2 q_2 + a_4 q_4 + \dots)\alpha}\\
    &\leq \fract{(q_1 - 1) \alpha} + \fract{-q_1 \alpha}
    = 1 - \fract{\alpha}.
\end{align*}
Thus, $\fract{n\alpha} + \fract{\alpha}<1$ and so $ \fract{(n+1)\alpha} = \fract{n\alpha} + \fract{\alpha} > \fract{n\alpha}$, as desired.
\end{proof}

\subsection{One-sided approximations}

For real $\xi$, let us define the distance to the nearest integer on the left, and the distance to the nearest integer on the right:
\begin{equation*}
    \absleft{\xi} := \fract{\xi}
    \quad \text{and} \quad
    \absright{\xi} := \fract{-\xi}.
\end{equation*}
    
The overload of notation is justified by the fact that 
we will go on to focus on one of $\absleft{\cdot}$, $\absright{\cdot}$ at a time; see also the lemma below.
Moreover, in Section~\ref{sec:proof} we will consider the functions $\fleft, \fright, \xleft, \xright$, and the definition above will be compatible with those functions.
Note that if $\{\absside{\cdot}, \abssideother{\cdot}\} = \{\absleft{\cdot}, \absright{\cdot}\}$, we have
\[
    \absside{-\xi} = 1 - \absside{\xi} = \abssideother{\xi},
    \quad \text{for all real non-integers } \xi.
\]
Moreover, we have the same ``addition rules'' as for $\fract{\cdot}$:
\begin{align*}
    \absside{\xi_1} + \absside{\xi_2} < 1
        \quad \iff \quad
        \absside{\xi_1 + \xi_2} = \absside{\xi_1} + \absside{\xi_2};\\
    \absside{\xi_1} \geq \absside{\xi_2} 
    \quad \iff \quad
    \absside{\xi_1 - \xi_2} = \absside{\xi_1} - \absside{\xi_2}.
\end{align*}  
We will also use the fact that
\[
    \absside{\xi}<\absside{\xi + \xi_1} < \absside{\xi + \xi_2} 
    \quad \implies \quad
    \absside{\xi_1}<\absside{\xi_2}.
\]

We start with a simple lemma for the situation where $\absside{n\alpha}>1/2$.

\begin{lemma}\label{lem:absside_larger12}
Let $\absside{\cdot} \in \{\absleft{\cdot}, \absright{\cdot}\}$.
If $\absside{n\alpha} > 1/2$, then either $\absside{(n-1)\alpha}< \absside{n\alpha}$ or $\absside{(n+1)\alpha}<\absside{n\alpha}$.
\end{lemma}
\begin{proof}
Since either $\absside{\alpha}<1/2$ or $\absside{-\alpha} <1/2$, we have that $\absside{\pm \alpha}< 1/2 < \absside{n\alpha}$ for one choice of $\pm$.
Then $\absside{(n\mp 1)\alpha} = \absside{n\alpha} - \absside{\pm \alpha} < \absside{n\alpha}$.
\end{proof}

Below, we will focus on one of $\absleft{\cdot}$, $\absright{\cdot}$ at a time. In fact, we will usually fix an integer $n$ and set
\begin{equation}\label{eq:def_absside}
    \absside{\cdot} := \begin{cases}
    	\absleft{\cdot}, 
			&\text{if } \dist{n \alpha}= \fract{n\alpha};\\
		\absright{\cdot}, 
			&\text{if } \dist{n\alpha}= 1 - \fract{n\alpha}.    
    \end{cases}
\end{equation}
In other words, we choose $\absside{\cdot}$ such that $\absside{n\alpha} = \dist{n \alpha}<1/2$.

We rephrase Lemma~\ref{lem:Ostr_direction} in terms of $\absside{\cdot}$:

\begin{lemma}\label{lem:absside_parity}
Let $\alpha \in (0,1/2)$ and $n,n'\geq 1$ with $k_0(n), k_0(n')\geq 1$.
Set $\absside{\cdot}$ as in \eqref{eq:def_absside}. 
Then
\[
    \absside{n'\alpha}<1/2
    \quad \iff \quad
    k_0(n)\equiv k_0(n') \pmod{2}.
\]
\end{lemma}

The next two lemmas describe how we can compare $\absside{n\alpha}$ and $\absside{n'\alpha}$ by looking at the first nonzero digits in $n,n'$.

\begin{lemma}\label{lem:dist_smallestConv}
Let $\alpha \in (0,1/2)$ and $n \geq 1$ and set $\absside{\cdot}$ as in \eqref{eq:def_absside}.
Then for all $n'\geq 1$ we have
\[
    k_0(n') < k_0(n)
    \quad \implies \quad
    \absside{n \alpha} < \absside{n' \alpha}.
\]
\end{lemma}
\begin{proof}
If $k_0(n) = 0$ , the implication is trivial.

Assume next that $k_0(n)\geq 2$.
Then if $k_0(n') = k_0(n) - 1$, we have from Lemma~\ref{lem:absside_parity} that $\absside{n'\alpha}>1/2>\absside{n\alpha}$. If $k_0(n')\leq k_0(n)-2$, then we obtain $\absside{n'\alpha}>\absside{n\alpha}$ from Lemma~\ref{lem:Ost_largedigits}(\ref{it:dist_lb}, \ref{it:dist_ub}).

We are left with the case $k_0(n) = 1$.
Lemma~\ref{lem:Ostr_direction} implies that $\absside{\cdot} = \absright{\cdot}$.
Assume $k_0(n') = 0$.
If $\fract{n'\alpha}<1/2$, then $\absright{n'\alpha}>1/2>\absright{n\alpha}$, as desired.
If $\fract{n'\alpha} > 1/2$, then as in the proof of Lemma~\ref{lem:q0-wrongside-improve} we see that $\fract{n'\alpha}< 1- \fract{\alpha}$, and so $\absright{n'\alpha}> \dist{\alpha} = \dist{q_0\alpha} > \dist{n\alpha} = \absright{n\alpha}$, where we used Lemma~\ref{lem:Ost_largedigits}(\ref{it:dist_ub}).
\end{proof}

\begin{lemma}\label{lem:dist_smallestDig}
Let $\alpha \in (0,1/2)$ and $n,n'\geq 1$ with $k_0(n) = k_0(n') = L$.
Assume $L \geq 1$ and set $\absside{\cdot}$ as in \eqref{eq:def_absside}.
Then 
\[
    b_L(n) < b_L(n')
    \quad \implies \quad
    \absside{n \alpha} < \absside{n' \alpha}.
\]
\end{lemma}
\begin{proof}
By our assumptions and by Lemma~\ref{lem:absside_parity}, we have $\absside{n\alpha} = \dist{n\alpha}$ and $\absside{n'\alpha} = \dist{n'\alpha}$.
Thus, the inequality is the statement (\ref{it:dist_compare}) in Lemma~\ref{lem:Ost_largedigits}.
\end{proof}

Now we state the best approximation property of (semi-)convergents in the one-sided setting.

\begin{lemma}\label{lem:semiconv-best-approx}
Let $\alpha \in (0,1/2)$ and $L\geq 0$.
Let $\absside{\cdot} \in \{\absleft{\cdot}, \absright{\cdot}\}$ be such that $\absside{q_L\alpha} < 1/2$.
Then 
\begin{equation}\label{eq:semiconv_ineq}
    \absside{q_L\alpha}
    >\absside{(q_L + q_{L+1}) \alpha} 
    > \dots 
    > \absside{(q_L + (a_{L+2} - 1)q_{L+1}) \alpha}
    > \absside{q_{L+2} \alpha}.
\end{equation}
Moreover, for $0\leq a \leq a_{L+2}-1$ we have 
\begin{equation}\label{eq:semiconv-best-approx}
    0 < q < q_L + (a+1) q_{L+1} 
    \quad \implies \quad 
    \absside{q\alpha} \geq \absside{(q_L + a q_{L+1})\alpha }.
\end{equation}
\end{lemma}
\begin{proof}
The inequalities in \eqref{eq:semiconv_ineq} follow from
\eqref{eq:qk_parity} and the fact that $a_{L+2}\dist{q_{L+1}\alpha}<\dist{q_{L}\alpha}$ (which follows, for example, from \eqref{eq:sum-everyother-full-convergents}).

For the best approximation property \eqref{eq:semiconv-best-approx}, note that all $q$ in the range $0<q<q_{L+2}$ which are not of the shape $q_L + a q_{L+1}$ with $0\leq a \leq a_{L+2}-1$, by \eqref{eq:greedy}, must either have $k_0(q)= L+1$, or $k_0(q) = L$ and $b_L(q) \geq 2$, or $k_0(q)\leq L-1$.

If $k_0(q) = L+1$, then Lemma~\ref{lem:Ostr_direction} implies $\absside{q\alpha}>1/2>\absside{q_L\alpha}$.

If $k_0(q)\leq L-1$, then Lemma~\ref{lem:dist_smallestConv} implies
$\absside{q\alpha}>\absside{q_L\alpha}$.

If $k_0(q) = L$ and $b_L(q)\geq 2$, then in the case $L\geq 1$, Lemma~\ref{lem:dist_smallestDig} implies $\absside{q\alpha}>\absside{q_L\alpha}$. 
If $L = 0$, this is easy to see as well.

Overall, $\absside{q\alpha}> \absside{q_L\alpha}$ holds for all $q$ that are not of the shape $q_L + a q_{L+1}$, and so the implication \eqref{eq:semiconv-best-approx} follows from \eqref{eq:semiconv_ineq}.
\end{proof}

\subsection{Minimising distances in certain ranges of integers}

The lemmas in this subsection will be particularly useful for characterising balancedness.

\begin{lemma}\label{lem:add_next_q}
Let $\alpha \in (0,1/2)$ and assume that $k_0(n) = L$ with $L\geq 1$.
Set $\absside{\cdot}$ as in \eqref{eq:def_absside}.    
Then
\[
    \absside{(n+q_{L+1})\alpha}
    < \absside{n\alpha}.
\]
More generally, the inequality holds for all $n$ with $k_0(n)\leq L$, as long as 
$\absside{n\alpha}<1/2$ and $\absside{q_L\alpha}<1/2$.
\end{lemma}
\begin{proof}
This follows from the fact that $\absside{n\alpha} = \dist{n\alpha}>\dist{q_{L+1}\alpha}$ (by Lemma~\ref{lem:Ost_largedigits}(\ref{it:dist_lb})) and $\absside{q_{L+1}\alpha} = 1-\dist{q_{L+1}\alpha}$ (by Lemma~\ref{lem:absside_parity}, since $L$ and $L+1$ have distinct parities). 
\end{proof}

In the cases where $L$ is too small for the above lemma, we will use the next two lemmas.

\begin{lemma}\label{lem:add_next_q_small}
Assume that $k_0(n) = 0$ and set $\absside{\cdot}$ as in \eqref{eq:def_absside}.
Then we have 
\[
    \absside{(n-q_1)\alpha}< \absside{n\alpha}
    \quad \text{or} \quad
    \absside{(n+q_1)\alpha} < \absside{n\alpha}.
\]
\end{lemma}
\begin{proof}
This follows from the fact that $\absside{n\alpha} = \dist{n\alpha}>\dist{q_{1}\alpha}$, by Lemma~\ref{lem:Ost_largedigits}(\ref{it:dist_lb}).
\end{proof}

\begin{lemma}\label{lem:dist_pm1}
Assume $b_0(n) \geq 2$ and
set $\absside{\cdot}$ as in \eqref{eq:def_absside}. 
Then we have
\[
    \absside{(n-1)\alpha}< \absside{n\alpha}
    \quad \text{or} \quad
    \absside{(n+1)\alpha}< \absside{n\alpha}. 
\]
\end{lemma}
\begin{proof}
Similarly to the proof of Lemma~\ref{lem:q0-wrongside-improve}, one can use \eqref{eq:delta_sign} and \eqref{eq:sum-everyother-full-convergents} to show that $\fract{\alpha} < \fract{n\alpha} < 1- \fract{\alpha}$.
This implies 
$\fract{(n-1)\alpha} < \fract{n\alpha} < \fract{(n+1)\alpha}$, and so $\absleft{(n-1)\alpha}<\absleft{n\alpha}$ and $\absright{(n+1)\alpha}<\absright{n\alpha}$. This proves the lemma.
\end{proof}

\begin{lemma}\label{lem:absside_qM}
Assume that either $k_0(n) \geq L$ or 
$n = q_{L-1} + \sum_{k = L+2t}^N b_k q_k$ with $t\geq 0$ and $b_{L+2t}\neq 0$. 
Set $\absside{\cdot}$ as in \eqref{eq:def_absside}.
Then 
\begin{equation}\label{eq:n-qL-twosides}
    \absside{n\alpha} 
    	= \min \{ \absside{x\alpha} \colon n - q_L < x < n + q_L\}.
\end{equation}
\end{lemma}
\begin{proof}
Let $n$ be as in the statement of the lemma. If $L = 0$, then $q_L = 1$, so the statement is trivial. 
Otherwise, by Lemma~\ref{lem:Ost_largedigits}(\ref{it:dist_ub}, \ref{it:dist_specialsmall}), we have $\dist{n\alpha}< \dist{q_{L-1}\alpha}$. 
On the other hand, by the best approximation property of convergents \eqref{eq:best_approx_prop}, we have $\dist{d\alpha}\geq \dist{q_{L-1}\alpha}$ for all $1\leq d<q_L$. 
Thus, for all $1 \leq d<q_L$ with $\absside{d\alpha} < 1/2$, we have
$\absside{(n + d)\alpha} = \absside{n\alpha} + \absside{d\alpha} > \absside{n\alpha}$ and $\absside{(n-d)\alpha} > 1/2 >\absside{n\alpha}$.
Similarly, for all $1 \leq d<q_L$ with $\absside{d\alpha} > 1/2$, we have
$\absside{(n+d)\alpha} > 1/2 >\absside{n\alpha}$ and $\absside{(n - d)\alpha} = \absside{n\alpha} + \dist{d\alpha} > \absside{n\alpha}$.
This implies \eqref{eq:n-qL-twosides}.
\end{proof}

\begin{lemma}\label{lem:absside_above}
Let $\alpha \in (0,1)$ and assume that $k_0(n)\geq L$ with $L\geq 1$.
Set $\absside{\cdot}$ as in \eqref{eq:def_absside}.
Then 
\[
    \absside{n\alpha} 
    	= \min \{ \absside{x\alpha} \colon n \leq x < n + q_{L-1} + q_{L}\}.
\]
Moreover, if $b_L(n) \leq a_{L+1}-1$, we can extend the range to
\[
    \absside{n\alpha} 
    	= \min \{ \absside{x\alpha} \colon n \leq x < n + q_{L-1} + 2q_{L}\}.
\]
\end{lemma}
\begin{proof}
Assume first that $k_0(n) = L$ with $L\geq 1$ and $\absside{n\alpha}<1/2$.
Then Lemma~\ref{lem:Ost_largedigits}(\ref{it:dist_ub}) implies $\absside{n\alpha}=\dist{n\alpha} < \dist{q_{L-1}\alpha}$.
Let $\abssideother{\cdot}$ denote the distance to the nearest integer in the opposite direction, i.e., $\{\absside{\cdot}, \abssideother{\cdot}\} = \{\absleft{\cdot}, \absright{\cdot}\}$.
Then by Lemma~\ref{lem:absside_parity} we have $\absside{q_L\alpha}<1/2$ and $\abssideother{q_{L-1}\alpha}<1/2$.

For the sake of contradiction, assume $\absside{(n+d)\alpha}< \absside{n\alpha}$ for some $1\leq d < q_{L-1} + q_L$. This is equivalent to $\abssideother{(n+d)\alpha}>\abssideother{n\alpha}$, and so
\[
    \abssideother{d\alpha}
    = \abssideother{(n+d)\alpha} - \abssideother{n\alpha}
    < 1 - \abssideother{n\alpha}
    = \absside{n\alpha}
    < \dist{q_{L-1}\alpha}
    = \abssideother{q_{L-1}\alpha}
\]
for some $1\leq d < q_{L-1} + q_L$, contradicting Lemma~\ref{lem:semiconv-best-approx} with $a = 0$.

If $k_0(n) > L$, then the previous case just gives us a stronger result than necessary.

Finally, assume that $k_0(n) = L$ and $b_L \leq a_{L+1}-1$.
Then Lemma~\ref{lem:Ost_largedigits}(\ref{it:dist_ub2}) says that in fact $\absside{n\alpha} < \dist{(q_{L-1} + q_L)\alpha}$, and we get the better bound by the same argument as before, this time using Lemma~\ref{lem:semiconv-best-approx} with $a = 1$.
\end{proof}

Next, we determine integers larger than $n$ which minimise $\absside{x\alpha}$ in certain ranges strictly above $n$.

\begin{lemma}\label{lem:second_best}
Let $\alpha \in (0,1)$ and $L\geq 0$, and assume $\absside{q_L\alpha}<1/2$.
Let $n\geq 1$ and
\[
    n' := \min\{x > n \colon \absside{x\alpha} < \absside{n\alpha} \}.
\]
Moreover, let $0\leq a \leq a_{L+2} -1$ and assume $n + q_L + a q_{L+1} < n'$.
Set
\[
    u : = \min\{n', n + q_L + (a+1)q_{L+1} \}.
\]
Then we have
\[
	\absside{(n+q_L+aq_{L+1})\alpha}
		=\min_{n<x<u} \absside{x\alpha}.
\]
\end{lemma}
\begin{proof}
First, note that the assumption $n + q_L + a q_{L+1} < n'$ implies $\absside{(n + q_L + a q_{L+1})\alpha}> \absside{n\alpha}$.
For the sake of contradiction, assume there exists an integer $x$ with $n < x < u$ such that $\absside{x\alpha} < \absside{(n+q_L+aq_{L+1})\alpha}$. 
Then, since $u\leq n'$ and $u \leq n+q_L+(a+1)q_{L+1}$, we have in fact 
\begin{equation}\label{eq:nd_1}
    \absside{n\alpha} < 
    \absside{x\alpha}
    = \absside{(n+d)\alpha}
    < \absside{(n+q_L+aq_{L+1})\alpha}
\end{equation}
with $0<d<q_L + (a+1)q_{L+1}$.
But \eqref{eq:nd_1} implies $\absside{d\alpha}< \absside{(q_L+aq_{L+1})\alpha}$, contradicting Lemma~\ref{lem:semiconv-best-approx}.
\end{proof}

\subsection{Ostrowski representations of $n$ and $q_T - n$}\label{sec:Ostr_qT-n}

Finally, we want to describe how the shape of the Ostrowski representations of $n$ and  $q_T -n $ are related to each other, for $T$ sufficiently large.

\begin{lemma}\label{lem:flip-simple}
Let $\alpha \in (0,1/2)$ and assume that $k_0(n) = L$ with $L\geq 1$
and $q_L < n < q_{N+1}$,  but that $n$ is not of the shape $n = q_L + \sum_{k=L+1+2t}^N b_k q_k$ with $t\geq 0$ and $b_{L+1+2t}>0$.

Then for $T\geq N+2$ we have that $q_T - n$ is of the shape $q_T - n = q_{L-1} + \sum_{k=L+2t}^{T-1} b'_k q_k$ with $t\geq 0$ and $b'_{L+2t}>0$.
\end{lemma}
\begin{proof}
Let $n$ be as in the statement of the lemma. Then by Lemma~\ref{lem:Ost_largedigits}(\ref{it:dist_notspecial}) we have $\dist{n\alpha} \geq \dist{q_L\alpha} + \dist{q_N} > \dist{q_L\alpha} + \dist{q_T \alpha}$, and so $\dist{(q_T - n)\alpha} >  \dist{q_L\alpha}$.
Now Lemma~\ref{lem:Ost_largedigits}(\ref{it:dist_ub}) implies that $k_0(q_T - n) \leq L$.
Note that $\fract{(q_T - n)\alpha}<1/2 \iff \fract{n\alpha}>1/2$. Thus, by Lemma~\ref{lem:Ostr_direction}, $k_0(n)$ and $k_0(q_T - n)$ must have distinct parities, and so in fact $k_0(q_T - n) \leq L-1$. 

Finally, since $k_0(n) = L$, Lemma~\ref{lem:Ost_largedigits}(\ref{it:dist_ub_N}) tells us that $\dist{n\alpha} < \dist{q_{L-1}\alpha} - \dist{q_{N+2}} \leq \dist{q_{L-1}\alpha} - \dist{q_T}$, and so
$\dist{(q_T - n)\alpha} < \dist{q_{L-1}\alpha}$.
Thus, we get from Lemma~\ref{lem:Ost_largedigits}(\ref{it:dist_lb}, \ref{it:dist_specialsmall})
that $q_T - n$ is of the shape $q_{L-1} + \sum_{k=L+2t}^{T-1} b'_k q_k$ with $t\geq 0$ and $b'_{L+2t}>0$.
\end{proof}

In preparation for our second lemma, we introduce some more notation:
For $M\geq 0$ let us define the ``small part of $n$'', obtained by discarding terms with indices larger than $M$:
\[
    n^{[\leq M]}
        := \sum_{k = 0}^M b_k(n) q_k.
\]
Note that, by \eqref{eq:greedy}, we have 
\begin{equation}\label{eq:nleqM-bound}
    n^{[\leq M]}<q_{M+1}.
\end{equation}
Similarly, we define the ``large part of $n$'':
\[
    n^{[\geq M]}
        := \sum_{k = M}^\infty b_k(n) q_k.
\]
Moreover, let 
\[
    k_{\geq M}(n) := \min \{k \colon k\geq M \text{ and } b_k(n) > 0\}.
\]
Note that this generalises our previous definition of $k_0(n) = k_{\geq 0}(n)$.

\begin{lemma}\label{lem:flip_parity}
Assume that $n^{[\leq M-1]}>0$ and $n^{[\geq M]}>0$.
Then if $q_{T-1} > n$, we have
\[
    k_{\geq M}(n) \equiv k_{\geq M}(q_T-n) \pmod{2}.
\]
\end{lemma}
\begin{proof}
Note that the assumption $n^{[\leq M-1]}>0$ implies that $M\geq 1$. Moreover, assume $n < q_{N+1}$ and $T\geq N+2$.

\caseI{1}
$k_{\geq M}(n) \equiv M+1 \pmod{2}$.
In particular, we have $b_M(n) = 0$ in this case, and we can write 
\[
    n = n^{[\leq M-1]} + n^{[\geq M+1]},
\]
and we have
\[
    k_0(n^{[\geq M+1]}) = k_{\geq M}(n) \equiv M+1 \pmod{2}.
\]
Let us write
\begin{align*}
    q_T - n
    = \underbrace{(q_T - n^{[\geq M+1]} - q_{M})}_{:= \nlarge'} + \underbrace{(q_{M} - n^{[\leq M-1]})}_{:= \nsmall'}.
\end{align*}
First, note that $0 < \nsmall' < q_M$. 
(This follows from the assumption $n^{[\leq M-1]} > 0$ and \eqref{eq:nleqM-bound}.)
Our goal is to show that $k_0(\nlarge')\geq M$ and  $k_0(\nlarge') \equiv M+1 \pmod{2}$.

Choose $\absside{\cdot} \in \{\absleft{\cdot}, \absright{\cdot}\}$ so that $\absside{q_{M+1}\alpha}<1/2$.
Then by the case assumption and Lemmas~\ref{lem:Ost_largedigits}(\ref{it:dist_ub_N}) and \ref{lem:Ostr_direction}, we have
\[
    \absside{n^{[\geq M+1]}\alpha} 
    = \dist{n^{[\geq M+1]}\alpha}  
    < \dist{q_M\alpha}- \dist{q_{N+2}\alpha} 
    \leq \dist{q_M\alpha}- \dist{q_{T}\alpha}.
\]
Since, moreover, $\dist{q_T\alpha} < \absside{n^{[\geq M+1]}\alpha}$, we get that 
\[
    \absside{(-q_T + n^{[\geq  M+1]})\alpha} 
    < \dist{q_M\alpha}
    = 1 - \absside{q_M\alpha}.
\]
This implies
\begin{align*}
    1 - \dist{q_M\alpha}
    =\absside{q_M\alpha}
    &< \absside{q_M\alpha} + \absside{(-q_T + n^{[\geq  M+1]})\alpha}\\
    &= \absside{(-q_T + n^{[\geq M+1]} + q_M)\alpha}<1,
\end{align*}
and so
\[
    \absside{\nlarge'\alpha}
    = \absside{(q_T - n^{[\geq M+1]} - q_M)\alpha}
    < \dist{q_M\alpha}.
\]
The fact that $\dist{\nlarge'\alpha} = \absside{\nlarge'\alpha} < \dist{q_M\alpha}$ implies that $k_0(\nlarge') \geq M$. 
Moreover, since $\absside{q_{M+1}\alpha}<1/2$, Lemma~\ref{lem:absside_parity} implies that $k_0(\nlarge') \equiv M+1 \pmod{2}$.

Overall, $q_T -n = \nlarge' + \nsmall'$ has the correct shape, namely $k_{\geq M}(q_T -n) \equiv M+1 \pmod{2}$.

\caseI{2}
$k_{\geq M}(n) \equiv M \pmod{2}$.
We want to show that $k_{\geq M}(q_T - n) \equiv M \pmod{2}$ as well.
Assume the contrary, i.e.,
\begin{equation*}
    k_{\geq M}(q_T-n) \equiv M+1 \pmod{2}.
\end{equation*}

If $(q_T-n)^{[\leq M-1]} >0$, then we can apply the result from Case~1 to $q_T -n $: for $n'' := q_{T+2} - (q_T-n) = n + (q_{T+2} - q_T)$ we have 
\begin{equation}\label{eq:contr_1}
    k_{\geq M} (n'') \equiv M+1 \pmod{2}.
\end{equation}
From $q_{T+2} - q_T = a_{T+2}q_{T+1}$ we see that $n''$ and $n$ have exactly the same digits up the digit with index $T+1$. Therefore, the congruence \eqref{eq:contr_1} contradicts the case assumption.

We are left with the case $(q_T-n)^{[\leq M-1]} =0$.
Then $k_0(q_T - n) \geq M+1$.
Lemma~\ref{lem:Ost_largedigits}(\ref{it:dist_ub}) implies 
$\dist{(q_T -n)\alpha} < \dist{q_{M}\alpha}$, and so $\dist{n\alpha} < \dist{q_M\alpha} + \dist{q_T \alpha} \leq \dist{q_M\alpha} + \dist{q_{N+2} \alpha}$. 
But then Lemma~\ref{lem:Ost_largedigits}(\ref{it:dist_lb_N}) implies that $k_0(n)\geq M$, which contradicts the assumption in the lemma that $n^{[\leq M-1]}>0$.
\end{proof}

We combine the previous two lemmas in the way we will want to apply in the next section.

\begin{lemma}\label{lem:flip-appl}
Let $\alpha \in (0,1/2)$ and $q_M<n < q_{N-1}$ for some $M\geq 1$.
Assume that $b_M(n)>0$, and that $n$ is not of the shape $n = q_M + \sum_{k=M + 1 + 2t}^N b_k q_k$ with $t\geq 0$ and $b_{M+1+2t}>0$.
Then for $T\geq N+2$ we have $k_{\geq M}(q_T-n) \equiv M \pmod{2}$.
\end{lemma}
\begin{proof}
If $n^{[\leq M-1]} > 0$, then Lemma~\ref{lem:flip_parity} implies that 
$k_{\geq M}(q_T-n) \equiv k_{\geq M}(n) = M \pmod{2}$.
If $n^{[\leq M-1]} = 0$, then $k_0(n) = M$, and so Lemma~\ref{lem:flip-simple} implies that $k_{\geq M}(q_T-n) \equiv M \pmod{2}$.
\end{proof}

\section{Proof of the full characterisation}\label{sec:proof}

In this section, we prove our main result, namely the full characterisation of balanced rectangles of Sturmian sequences (Theorem~\ref{thm:bal_char_sturmian}).
As before, let $\alpha \in (0,1/2)$ be irrational and fix integers
$2\leq m \leq n$.
In view of Theorem~\ref{thm:def-equiv}, we can phrase Theorem~\ref{thm:bal_char_sturmian} in terms of balanced intervals. 
Moreover, as mentioned in Remark~\ref{rem:Ad_appl}, we can use Lemma~\ref{lem:bal_fg}, setting 
\begin{align*}
    B &:= \{ \xi_0, \xi_1, \ldots, \xi_{m-1} \}
    \quad \text{with} \quad
    \xi_\ell := \fract{\ell\alpha} \text{ for } 0 \leq \ell \leq m-1;\\
    \delta &:= \fract{n\alpha}.
\end{align*}
    
This means that we need to characterise the situations when $\fleft$ and $\fright$ from Definition~\ref{def:flr} are bijective (and we know from Lemma~\ref{lem:bal_fg} that one is bijective if and only if the other is bijective, so we can focus on either of them).

Recall that to describe $\fleft, \fright$, we need to find the closest points from the set $B$ on either side of $\xi_\ell + \delta \bmod 1 = \fract{
(n+\ell)\alpha}$. In other words, we need to subtract some positive integer $x$ from $n+\ell$, so that $n+\ell-x$ falls into $[0,m-1]$, and so that $\fract{x\alpha}$ causes a minimal shift to the left or to the right.

To formalise this, we define
\begin{align}
	\xleft(\ell) &:= n + \ell - \fleft(\ell), \label{eq:xleft_def} \\
	\xright(\ell) &:= n + \ell - \fright(\ell),\nonumber
\end{align} 
for $0\leq \ell \leq m-1$.
Note that since  $\fleft(\ell),\fright(\ell) \in [0, m-1]$, we have
\begin{equation}\label{eq:x_range}
	\xleft(\ell), \xright(\ell)
	\in [n+\ell-m+1, n+\ell]
	=: X(\ell)
\end{equation}
for $0\leq \ell \leq m-1$.

The next lemma captures the fact that $\fract{\xleft(\ell) \alpha}$ and $\fract{\xright(\ell) \alpha}$ must correspond to ``minimal shifts''.

\begin{lemma}\label{lem:x_min}
With the above definitions we have
\begin{align*}
	\fract{\xleft(\ell) \alpha} 
	&= \min_{x \in X(\ell)} \fract{x\alpha},\\
	\fract{\xright(\ell) \alpha} 
	&= \max_{x \in X(\ell)} \fract{x\alpha}.
\end{align*}
\end{lemma}
\begin{proof}
This follows directly from the definitions: for $\xleft(\ell)$ we have
\begin{align*}
	\fract{\xleft(\ell)\alpha}	&\stackrel{\eqref{eq:xleft_def}}{=} \fract{\ell\alpha + n\alpha - \fleft(\ell)\alpha}\\
						&\stackrel{\textnormal{Def.\  \ref{def:flr}}}{=} \min_{0\leq i \leq m-1} \fract{\ell\alpha + n\alpha - i\alpha}\\
	&= \min_{n+\ell-m+1 \leq x \leq n+\ell} \fract{x\alpha},
\end{align*}
and we can check the formula for $\fract{\xright(\ell) \alpha}$ analogously.
\end{proof}

\begin{rem}
Since $\fract{x\alpha}\neq \fract{y\alpha}$ for irrational $\alpha$ and integers $x\neq y$, Lemma~\ref{lem:x_min} gives us an alternative definition for the functions $\xleft,\xright$.
\end{rem}

Now we can phrase $\fleft,\fright$ being bijective in terms of $\xleft, \xright$.

\begin{lemma}\label{lem:x_bij}
Let $(\f,\x) = (\fleft,\xleft)$ or $( \fright, \xright)$.
Then $\f$ is bijective if and only if 
\[
	\x(\ell) - \ell
	\neq \x(\ell') - \ell'
\]
for all $0\leq \ell < \ell' \leq m-1$.
\end{lemma}
\begin{proof}
This follows directly from the definition: We have $\x(\ell) = n+\ell-\f(\ell)$, which gives us $\f(\ell) = n + \ell - \x(\ell)$, and so $\f(\ell) = \f(\ell') \iff \x(\ell) - \ell = \x(\ell')-\ell'$. Since $\f$ is a map on the finite set $\{0,1, \ldots, m-1\}$, it is  injective if and only if it is bijective.
Thus, $\f$ is bijective if and only if $\f(\ell)\neq \f(\ell')$ for all $\ell \neq \ell'$, and by the previous argument this is equivalent to $\x(\ell) - \ell \neq \x(\ell')-\ell'$ for all $\ell\neq \ell'$.
\end{proof}

A brief recap: We want to characterise all $2\leq m \leq n$ such that the $m\times n$ rectangles of the Sturmian words with slope $\alpha$ are balanced (prove Theorem~\ref{thm:bal_char_sturmian}).
The $m\times n$ rectangles of the Sturmian words with slope $\alpha$ being balanced is equivalent to the intervals of length $\fract{n\alpha}$ being balanced with respect to $(\alpha, m)$ (Theorem~\ref{thm:def-equiv}).
This again is equivalent to the function $\f = \fleft$ or $\f = \fright$ being bijective (Lemma~\ref{lem:bal_fg}).
And to decide whether $\f$ is bijective, we can use Lemma~\ref{lem:x_bij} and the corresponding function $\x= \xleft$ or $\x = \xright$, which can be defined via Lemma~\ref{lem:x_min}.
Indeed, this is our strategy.
In order to unify our arguments, recall that in the previous section we defined
\begin{equation*}
    \absleft{\xi} = \fract{\xi}
    \quad \text{and} \quad
    \absright{\cdot} = \fract{- \xi}.
\end{equation*}
In the rest of this section, we will always either set $(\f,\x,\absside{\cdot}) = (\fleft, \xleft, \absleft{\cdot})$ or $(\f,\x,\absside{\cdot}) = (\fright, \xright, \absright{\cdot})$.
In either case, we can now phrase Lemma~\ref{lem:x_min} as 
\[
	\absside{\x(\ell) \alpha}
    = \min_{x \in X(\ell)} \absside{x\alpha},
\]
where the range $X(\ell) = [n+\ell-m+1, n+\ell]$ was defined in \eqref{eq:x_range}.
In view of this and Lemma~\ref{lem:x_bij}, we are interested in integers $x$ in the range 
\[
    [n-m+1, n+m-1] = \bigcup_{0\leq \ell \leq m-1} X(\ell)
\]
that minimise $\absleft{x\alpha}$ or $\absright{x\alpha}$.

Let $\xleftstar, \xrightstar$ be the integers that minimise $\absleft{x\alpha}$ or $\absright{x\alpha}$, respectively, in the full range $[n-m+1, n+m-1]$. Again, according to the setting, we will write just $\xstar$ for $\xleftstar$ or $\xrightstar$.
In other words, in our notation we have
\[
    \absside{\xstar \alpha}
    = \min_{n-m+1 \leq x \leq n+m-1} \absside{x\alpha}.
\]

Note that the only integer that occurs in the range $X(\ell)= [n+\ell-m+1, n+\ell]$ for every $0\leq \ell \leq m-1$ is the integer $n$.
Therefore, the case $\xstar = n$ is particularly easy, and we start with this case.

\subsection{The case $n = \xstar$}

\begin{lemma}\label{lem:xstarequalsn_bij}
Let $(\f,\xstar) = (\fleft, \xleftstar)$ or $(\f,\xstar) = (\fright, \xrightstar)$.
If $\xstar = n$, then $\f$ is bijective. 
\end{lemma}
\begin{proof}
Since $\xstar = n \in X(\ell) = [n+\ell-m+1, n+\ell]$ for every $0\leq \ell \leq m-1$, we have in fact $\x(\ell) = \xstar$ for all $0\leq \ell \leq m-1$.
Thus, $\x(\ell) - \ell = \xstar - \ell \neq \xstar - \ell' = \x(\ell') - \ell'$ for all $\ell\neq \ell'$, and by Lemma~\ref{lem:x_bij} the function $\f$ is bijective. 
\end{proof}

We also characterise the case $\xstar = n$ in terms of the Ostrowski representations of $m,n$.

\begin{lemma}\label{lem:xstarequalsn_repr}
Let $2\leq m \leq n$ with $q_{M-1}<m\leq q_M$. 
We have $n\in \{\xleftstar, \xrightstar\}$ if and only if 
the Ostrowski representation of $n$ with respect to $\alpha$ is of one of the following two shapes:
\begin{enumerate}[label = (\alph*)]
\item\label{it:case_n_split} 
    $n = \sum_{k = L}^N b_k q_k$ with $L\geq M$;
\item\label{it:case_n_split2} 
$n = q_{M-1} + \sum_{k = M + 2t}^N b_k q_k$ with $t\geq 0$ and $b_{M+2t} \neq 0$.
\end{enumerate}
\end{lemma}
\begin{proof}
Let 
\[
    n = \sum_{k = L}^N b_k q_k
    \quad \text{with } b_L\neq 0.
\]
If $\absside{n\alpha} > 1/2$ for $\absside{\cdot} = \absleft{\cdot}$ or $\absside{\cdot} = \absright{\cdot}$, then, by Lemma~\ref{lem:absside_larger12}, one of $ \absside{(n-1)\alpha}, \absside{(n+1)\alpha}$ must be strictly smaller than $\absside{n\alpha}$. 
Since $n-1, n+1$ are both in the range $[n-m+1, n+m-1]$, we have $n\neq \xstar$. 
Therefore, it suffices to consider the choice of $\absside{\cdot} \in \{\absleft{\cdot}, \absright{\cdot}\}$ for which $\absside{n\alpha} < 1/2$, and check whether $\xstar = n$.
We systematically go through all possible representations of $n$.

\caseI{1}
Either $L\geq M$, or $L = M-1$ and $n = q_{M-1} + \sum_{k = M + 2t}^N b_k q_k$ with $t\geq 0$ and $b_{M+2t} \neq 0$. In other words, $n$ has the representation from \ref{it:case_n_split} or \ref{it:case_n_split2}. 
Then Lemma~\ref{lem:absside_qM} says that $n$ minimises $\absside{n\alpha}$ in the range $[n - q_M + 1, n+ q_M-1] \supseteq [n-m+1, n+m-1]$. Thus, we have indeed $n = \xstar$.

\caseI{2}
$L = M-1$ and $n = q_{M-1} + \sum_{k = M + 1 + 2t}^N b_k q_k$ with $t\geq 0$ and $b_{M+1+2t} \neq 0$.
Set $n_1 := \sum_{k = M + 1 + 2t}^N b_k q_k$.
By Lemma~\ref{lem:Ostr_direction} 
we have $\absside{n_1\alpha} < 1/2$ as well, except possibly if $M-1 = 0$. 
In the exceptional case (i.e., if $\fract{n\alpha}$ is ``on the wrong side of $1/2$''), we can apply Lemma~\ref{lem:q0-wrongside-improve} and obtain $\absside{(n+1)\alpha} < \absside{n\alpha}$.
Otherwise, if $\absside{n_1\alpha} < 1/2$, then since $q_{M-1}< m$, the integer $n_1$ is in the range $[n-m+1,n+m-1]$. By Lemma~\ref{lem:dist_smallestConv}, we have $\absside{n_1\alpha} < \absside{n\alpha}$, and so $n\neq \xstar$.

\caseI{3}
$L = M-1$ and $n = b_{M-1} q_{M-1} + \sum_{k = M}^N b_k q_k$ with $b_{M-1}\geq 2$.
If $L\geq 1$, we set $n_1 := n - q_{M-1} > n-m$. Then, by Lemma~\ref{lem:dist_smallestDig}, we have $\absside{n_1\alpha}< \absside{n\alpha}$, and so $n\neq \xstar$.
If $L = 0$, the same follows from Lemma~\ref{lem:dist_pm1}.

\caseI{4} 
$L\leq M-2$. If $L\geq 1$, we set $n_1 := n + q_{L+1} < n + m$. 
Then by Lemma~\ref{lem:add_next_q}, we have $\absside{n_1\alpha}< \absside{n\alpha}$, and so $n\neq \xstar$.
If $L = 0$, the same follows from Lemma~\ref{lem:add_next_q_small}.
\end{proof}

The two above lemmas give us a partial result towards the full characterisation in Theorem~\ref{thm:bal_char_sturmian}; note that the cases \ref{it:case_n_split} and \ref{it:case_n_split2} in Lemma~\ref{lem:xstarequalsn_repr} are almost the same as the cases \ref{it:orig_split1} and \ref{it:orig_split2} in Theorem~\ref{thm:bal_char_sturmian}.
Next, we provide some lemmas that will be useful in the situation where $\xstar \neq n$. 
We start with two lemmas which will, roughly speaking, allow us to assume $\xstar < n$ without loss of generality.

\subsection{Switching between $n$ and $q_T - n$}

\begin{lemma}\label{lem:qT-intervals}
Let $2\leq m\leq n$. Then, for sufficiently large $T$, the intervals of length $\fract{n\alpha}$ are balanced with respect to $(\alpha, m)$ if and only if the intervals of length $\fract{(q_T - n) \alpha}$ are balanced with respect to $(\alpha, m)$.
\end{lemma}
\begin{proof}
This follows directly from Lemma~\ref{lem:alternative_intervals} and the fact that $\dist{q_T\alpha}$ gets arbitrarily small for large $T$.
\end{proof}

\begin{lemma}\label{lem:-n_xstar}
Let $\xleftstar(m,n) = \xleftstar$ be defined as before, and define $\xrightstar(m,q_T -n)$ in an analogous way, i.e., $\xrightstar(m,q_T -n)$ minimises $\absright{x\alpha}$ in the range $[q_T - n -m+1, q_T - n + m -1]$.
Then for sufficiently large $T$ we have 
\[
    \xleftstar(m,n) > n \iff \xrightstar(m,q_T-n) < q_T - n.
\]
\end{lemma}
\begin{proof}
Note that since $\absleft{\xi} = \absright{-\xi}$, we have
\[
    \absright{-\xleftstar \alpha}
    = \absleft{\xleftstar \alpha}
    = \min_{n-m+1 \leq x \leq n+m-1} \absleft{x \alpha}
    = \min_{-(n+m-1) \leq x \leq -(n-m+1)} \absright{x\alpha}.
\]
If $q_T$ is sufficiently large, adding $q_T$ to every $x$ in the range $[-(n+m-1) , -(n-m+1)] = [-n-m+1, -n+m-1]$ does not change the ordering of the numbers $\absright{x\alpha}$, and so the above equation implies
\[
    \absright{(q_T - \xleftstar)\alpha}
    = \max_{q_T-n-m+1 \leq x \leq q_T-n+m-1} \absright{x\alpha}.
\]
This means that $q_T - \xleftstar = \xrightstar(m,q_T-n)$, and so 
$\xleftstar>n \iff \xrightstar(m,q_T-n) < q_T - n$.
\end{proof}

\subsection{Technical lemmas for the case $\xstar < n$}

If $\xstar < n$, then we do not have $\x(\ell) = \xstar$ for all $\ell \in [0,m-1]$. However, it is not hard to see that $\x(\ell) = \xstar$ for a certain range of $\ell$'s (which, admittedly, might be only the single integer $\ell = n-m+1$).
We can say more about the precise values of $\x(\ell)$ in adjacent ranges.
This is probably the most technical part of the paper; after that we will be able to prove (non-)balancedness in various cases.

Recall that we always implicitly assume either $(\f,\x,\absside{\cdot},\xstar) = (\fleft, \xleft, \absleft{\cdot},\xleftstar)$ or $(\f,\x,\absside{\cdot},\xstar) = (\fright, \xright, \absright{\cdot},\xrightstar)$.
For technical reasons, we assume $m\geq q_1$; the cases where $m$ is very small are actually quite easy and we will deal with them separately in Section~\ref{sec:very-small-m}.

\begin{lemma}\label{lem:x-ranges}
Assume $2\leq q_1 \leq m\leq n$ and $q_{M-1} < m \leq q_M$.
Moreover, assume that $\xstar < n$.
Then we have
\[
    \x(\ell) = \xstar 
    \quad \iff \quad 
    \ell \in [0, m - 1 + \xstar - n] 
    =: R_0 .
\]
If $\absside{q_{M-1}\alpha}<1/2$, then
\begin{align*}
    &\x(\ell) = \xstar + q_{M-1}\\
    &\iff \quad 
    \ell \in [m + \xstar - n,  \min\{m + \xstar - n -1 + q_{M-1}, m-1\}]
    =: R_1.
\end{align*}
If $\absside{q_{M}\alpha}<1/2$ and $m = q_M$, then
\begin{align*}
    &\x(\ell) = \xstar + q_{M}
    \quad \iff \quad 
    \ell \in [m + \xstar - n,  m-1]
    =: R_1'.
\end{align*}
If $\absside{q_{M}\alpha}<1/2$ and
$q_{M-2} + a q_{M-1} \leq m < q_{M-2} + (a+1) q_{M-1}$ 
with $0\leq a \leq a_M-1$,
then
\begin{align*}
    &\x(\ell) = \xstar + q_{M-2} + a q_{M-1}\\
    &\qquad \text{for all }
    \ell \in [m + \xstar - n,  \min\{ \xstar - n + q_{M-2} + (a+1) q_{M-1} - 1, m-1\}]
    =: R_1'';\\
    &\x(\ell) = \xstar + q_{M-2} + (a+1) q_{M-1}\\
    &\qquad \text{for all }
    \ell \in [\xstar - n + q_{M-2} + (a+1) q_{M-1}, \\
    &  \qquad \qquad \qquad \qquad \min\{ \xstar - n + q_{M-2} + (a+2) q_{M-1} - 1, m-1\}]
    =: R_2''.
\end{align*}
Note that the range $R_2''$ is empty if and only if $m-1 \leq  \xstar - n + q_{M-2} + (a+1) q_{M-1}-1$.
\end{lemma}
\begin{proof}
By the definitions of $\xstar$ and $\x(\cdot)$, we clearly have $\x(\ell) = \xstar$ if and only if $\xstar \in X(\ell) = [n+\ell-m+1, n+\ell]$. Under our assumption $\xstar < n$, this happens exactly for $\ell \in [0, m - 1 + \xstar - n]$. This settles the statement regarding the range $R_0$.
Now we go through all the cases from the statement of the lemma.

\caseI{1}
$\absside{q_{M-1}\alpha}<1/2$.

First, note that $\xstar + q_{M-1} \in X(\ell)= [n+\ell-m+1, n+\ell]$ if and only if $\ell \in [\xstar - n + q_{M-1}, m + \xstar - n -1 + q_{M-1}]$.
Since $\xstar-n+q_{M-1} \leq m + \xstar -n$, we have indeed $\xstar + q_{M-1} \in X(\ell)$ for all $\ell \in R_1 = [m + \xstar - n,  \min\{m + \xstar - n -1 + q_{M-1}, m-1\}]$ (and not for any larger $\ell$).
Now we only need to show that $\xstar + q_{M-1}$ indeed minimises $\absside{x\alpha}$ in each range $X(\ell)$ for $\ell \in R_1$.
In other words, we need to show that
\begin{equation}\label{eq:wts_min_M-1}
    \absside{(\xstar + q_{M-1})\alpha}
    = \min \{ \absside{x\alpha} \colon x \in [\xstar +1, \min\{m + \xstar -1 + q_{M-1}, n+ m-1\} ]\}.
\end{equation}
From the definition of $\xstar$ it follows that
\[
    x' := \min\{x > \xstar \colon \absside{x\alpha} < \absside{\xstar \alpha} \}
    > n+m-1.
\]
Moreover, note that $m + \xstar -1 + q_{M-1} < \xstar + q_{M-1} + q_M$.
Therefore, we get \eqref{eq:wts_min_M-1} directly from Lemma~\ref{lem:second_best} with $a = 0$. This settles the statement regarding the range $R_1$.

\caseI{2}
$\absside{q_{M}\alpha}<1/2$ and $m = q_M$

First, note that $\xstar+ q_M \in X(\ell) = [n+\ell-m+1, n+\ell]$ if and only if $\ell \in [\xstar - n + q_{M}, m + \xstar - n -1 + q_{M}] = [m + \xstar - n, 2m + \xstar - n -1]$.
Since $2m + \xstar - n -1 \geq m$, we have in fact $\xstar+ q_M \in X(\ell)$ for all $\ell \in R_1' = [m + \xstar - n, m-1]$.
Now we only need to show that $\xstar + q_M$ indeed minimises $\absside{x\alpha}$ in each range $X(\ell)$ for $\ell \in R_1'$.
Since the largest element in all these ranges is $n+m-1 \leq \xstar + 2m - 2 = \xstar + 2q_M - 2 < \xstar + q_M + q_{M+1}$,
this follows again directly from Lemma~\ref{lem:second_best} with $a=0$.

\caseI{3}
$\absside{q_{M}\alpha}<1/2$ and 
$q_{M-2} + a q_{M-1} \leq m < q_{M-2} + (a+1) q_{M-1}$ 
with $0\leq a \leq a_M-1$.

First, note that $\xstar + q_{M-2} + a q_{M-1} \in X(\ell)= [n+\ell-m+1, n+\ell]$ if and only if $\ell \in [\xstar - n + q_{M-2} + a q_{M-1}, m + \xstar - n -1 + q_{M-2} + a q_{M-1}]$.
On the one hand, we have $\xstar-n+q_{M-2} + a q_{M-1} \leq m + \xstar -n$. On the other hand, recall that in the lemma we assume $m>q_{M-1}$, and so 
$m + \xstar - n -1 + q_{M-2} + a q_{M-1} \geq \xstar - n + q_{M-2} + (a+1) q_{M-1} - 1$.
Thus, we have indeed $\xstar + q_{M-2} + a q_{M-1} \in X(\ell)$ for
$\ell \in R_1'' = [m + \xstar - n,  \min\{ \xstar - n + q_{M-2} + (a+1) q_{M-1} - 1, m-1\}]$.
We need to check that $\xstar + q_{M-2} + a q_{M-1}$ minimises $\absside{x \alpha}$ in the range $[\xstar + 1,  \xstar + q_{M-2} + (a+1) q_{M-1} - 1]$. Indeed, this follows directly from Lemma~\ref{lem:second_best}.

The proof for the range $R_2''$ is completely analogous, except for the following detail: if $a+1 = a_{M}$, then $q_{M-2} + (a+1) q_{M-1} = q_M$, and Lemma~\ref{lem:second_best} would allow us to make $R_2''$ even larger.
\end{proof}

In the case $\xstar < n$ we can now
use the above lemma to characterise $\f$ being bijective in terms of the shape of $\xstar$ and $m$.
Again, we assume $m\geq q_1$.

\begin{lemma}\label{lem:xstar-parity-bij}
Assume $2 \leq q_1 \leq m\leq n$ and $q_{M-1} < m \leq q_M$.
Moreover, assume that $\xstar < n$.
Then $\f$ is bijective in exactly the three following cases:
\begin{itemize}
\item $\absside{q_{M}\alpha}<1/2$ and $m = q_M$;
\item $\absside{q_{M}\alpha}<1/2$, $m = q_{M-2} + a q_{M-1}$ for some $1 \leq a \leq a_{M}-1$ and $M-2\geq 0$, and $\xstar \geq n - q_{M-1}$;
\item $\absside{q_{M}\alpha}<1/2$, 
$q_{M-2} + a q_{M-1}< m < q_{M-2} + (a+1) q_{M-1}$ for some $0 \leq a \leq a_M-1$ and
$\xstar = n - q_{M-1}$.
\end{itemize}
\end{lemma}
\begin{proof}
We distinguish between five cases, according to whether $\absside{q_{M}\alpha}<1/2$ or $\absside{q_{M-1}\alpha}<1/2$, and some extra conditions.

\caseI{1}
$\absside{q_{M-1}\alpha}<1/2$. 

\noindent Then, by Lemma~\ref{lem:x-ranges}, we have $\x(\ell) = \xstar$ for $\ell \in R_0$ and $\x(\ell) = \xstar + q_{M-1}$ for $\ell \in R_1$.
Since $m-1 \geq q_{M-1}$ and $m + \xstar - n -1 + q_{M-1} \geq q_{M-1}$, the range $R_0 \cup R_1$ contains at least $q_{M-1}+1$ consecutive integers.
Therefore, there exist $\ell \in R_0$ and $\ell' \in R_1$ such that $\ell'-\ell = q_{M-1}$. For these $\ell, \ell'$ we have $\x(\ell') - \x(\ell) = q_{M-1} = \ell' - \ell$, and so $\f$ is not bijective by Lemma~\ref{lem:x_bij}.

\caseI{2}
$\absside{q_{M}\alpha}<1/2$ and $m = q_M$. 

\noindent By Lemma~\ref{lem:x-ranges}, we have either $\x(\ell) = \xstar$ or $\x(\ell) = \xstar + q_M$ for all $\ell \in R_0 \cup R_1' = [0,m-1] = [0, q_M - 1]$. Therefore, it is clear that we cannot have $\ell-\ell' = \x(\ell) - \x(\ell')$ for $\ell\neq \ell' \in [0, m-1]$, and so $\f$ is bijective by Lemma~\ref{lem:x_bij}.

\caseI{3}
$\absside{q_{M}\alpha}<1/2$ and $m = q_{M-2} + a q_{M-1}$ with $1\leq a \leq a_M -1$ and $\xstar \geq n - q_{M-1}$.

\noindent Then, by Lemma~\ref{lem:x-ranges}, we have $\x(\ell) = \xstar$ for $\ell \in R_0$ and $\x(\ell) = \xstar + q_{M-2} + a q_{M-1}$ for $\ell \in R_1''$.
Moreover, since we are assuming $\xstar \geq n- q_{M-1}$, we have
\[
    \xstar - n + q_{M-2} + (a+1) q_{M-1} - 1
    \geq q_{M-2} + a q_{M-1} - 1
    = m-1.
\]
In other words, $R_0 \cup R_1'' = [0,m-1]$, and so for all $\ell \in [0, m-1] = [0, q_{M-2} + a q_{M-1} - 1]$ either $\x(\ell) = \xstar$ or $\x(\ell) = \xstar + q_{M-2} + a q_{M-1}$.
As in Case 2, it is clear that $\f$ is bijective.

\caseI{4}
$\absside{q_{M}\alpha}<1/2$ and $m = q_{M-2} + a q_{M-1}$ with $1\leq a \leq a_M-1$, but now $\xstar < n - q_{M-1}$.

\noindent Then 
\[
     \xstar - n + q_{M-2} + (a+1) q_{M-1} - 1
     < m-1, 
\]
and so there is at least one $\ell \in R_2''$.
Now since the range $R_1''$ contains exactly $q_{M-1}$ integers (note that $m = q_{M-2} + a q_{M-1}$), the range $R_1'' \cup R_2''$ contains at least $q_{M-1}+1$ consecutive integers.
Moreover, Lemma~\ref{lem:x-ranges} says that $\x(\ell) = \xstar + q_{M-2} + a q_{M-1}$ for $\ell \in R_1''$ and $\x(\ell) = \xstar + q_{M-2} + (a+1) q_{M-1}$ for $\ell \in R_2''$. Thus, there must exist $\ell \in R_1''$ and $\ell' \in R_2''$ with $\ell' - \ell = q_{M-1} = \x(\ell') - \x(\ell)$, and so $\f$ is not bijective.

\caseI{5} 
$\absside{q_{M}\alpha}<1/2$ and $q_{M-2} + a q_{M-1}< m < q_{M-2} + (a+1) q_{M-1}$ for some $0 \leq a \leq a_M-1$.

If $\xstar > n - q_{M-1}$, then one can check that $R_0 \cup R_1'' \supseteq [0, q_{M-2} + a q_{M-1}]$. Therefore, there must exist $\ell \in R_0$ and $\ell' \in R_1''$ with $\ell' - \ell = q_{M-2} + a q_{M-1} = \x(\ell') - \x(\ell)$. Thus, $\f$ is not bijective.

If $\xstar < n - q_{M-1}$, then one can check that the range $R_1'' \cup R_2''$ contains at least $q_{M-1} + 1$ consecutive integers, 
and that $R_1''$, $R_2''$ are non-empty,
and so, as in Case 4, $\f$ is not bijective.

If $\xstar = n - q_{M-1}$, then the first term in the minimum in the upper bound of $R_2''$ is
\[
   \xstar - n + q_{M-2} + (a+2) q_{M-1} - 1
    = q_{M+2} + (a+1) q_{M-1} - 1
    \geq m-1,
\]
so $R_0 \cup R_1'' \cup R_2'' = [0,m-1]$. Also, one can check that $R_0 \cup R_1'' \not\supseteq [0, q_{M-2} + a q_{M-1}]$, and that $R_1'' \cup R_2''$ contains at most $q_{M-1}$ consecutive integers.
Thus, $\x(\ell') - \x(\ell) \neq \ell' - \ell$ for all $\ell \in R_0$ and $\ell' \in R_1''$, as well as for all $\ell \in R_1''$ and $\ell' \in R_2''$.
It is also easy to see that $\x(\ell') - \x(\ell) \neq \ell' - \ell$ for all $\ell \in R_0$ and $\ell \in R_2''$ because $m-1 < q_{M-2} + (a+1) q_{M-1}$.
Thus, $\f$ is bijective.
\end{proof}

\subsection{The cases where $\xstar \neq n$}

We now collect some of our results to show that ``usually'' we are in the not balanced situation, which is a big step towards finishing the proof of Theorem~\ref{thm:bal_char_sturmian}.

First, we show that the third special case from Lemma~\ref{lem:xstar-parity-bij} actually corresponds to the case where $n \in \{\xleftstar, \xrightstar\}$.

\begin{lemma}\label{lem:actually-nstar}
Let $2 \leq q_1 \leq m\leq n$ and $q_{M-1} < m < q_M$.
Assume that $\absside{q_M \alpha}<1/2$ and $\xstar = n - q_{M-1}$.
Then $n \in \{\xleftstar, \xrightstar\}$.
\end{lemma}
\begin{proof}
The assumptions $q_1 \leq m < q_M$ imply $M\geq 2$.
The assumption $\absside{q_M\alpha}<1/2$ implies that $k_0(\xstar) \equiv M \pmod{2}$ (or possibly $k_0(\xstar) = 0$).

If $k_0(\xstar) \geq M$, then it follows from the basic properties of Ostrowski representations that
$n = q_{M-1} + \xstar$ has one of the two shapes \ref{it:case_n_split}, \ref{it:case_n_split2} from Lemma~\ref{lem:xstarequalsn_repr}, which implies $n \in \{\xleftstar, \xrightstar\}$.

If $k_0(\xstar) \leq M-2$, then Lemma~\ref{lem:add_next_q} implies that $\absside{n\alpha} = \absside{(\xstar + q_{M-1})\alpha} < \absside{\xstar \alpha}$, contradicting the definition of $\xstar$.
\end{proof}

\begin{rem}
In view of Lemma~\ref{lem:actually-nstar}, if we assume $n\notin\{\xleftstar, \xrightstar\}$, then Lemma~\ref{lem:xstar-parity-bij} now only gives us two cases where $\f$ is bijective, namely where $m$ is either a convergent or a semi-convergent (and some extra condition).
\end{rem}

\begin{lemma}\label{lem:notbal_genericcase}
Let $2\leq q_1 \leq m \leq n$.
Assume that $q_{M-1} < m < q_M$ and that $m$ is not a semi-convergent.
Moreover, assume that $n\notin \{ \xleftstar, \xrightstar\}$.
Then $\fleft,\fright$ are not bijective.
\end{lemma}
\begin{proof}
Recall from Lemma~\ref{lem:bal_fg} that $\fleft$ being bijective is equivalent to $\fright$ being bijective, which is equivalent to the intervals of length $\fract{n\alpha}$ with respect to $(\alpha, m)$ being balanced.
Moreover, recall from Lemma~\ref{lem:qT-intervals} that the intervals of length $\fract{n\alpha}$ are balanced if and only if the intervals of length $\fract{(q_T-n)\alpha}$ are balanced, for $T$ sufficiently large. 

Now, if $\xleftstar < n$, then Lemmas~\ref{lem:xstar-parity-bij} and \ref{lem:actually-nstar} imply that $\fleft$ is not bijective.

If $\xleftstar>n$, then by Lemma~\ref{lem:-n_xstar} we have $\xrightstar(m,q_T - n) < q_T - n$. Then again Lemmas~\ref{lem:xstar-parity-bij} and \ref{lem:actually-nstar} imply that the intervals of length $\fract{(q_T -n)\alpha}$ are not balanced with respect to $(\alpha, m)$, and so neither are those of length $\fract{n\alpha}$.
\end{proof}

In the last two lemmas we deal with the cases where $m$ is a (semi-)convergent.
Note that $n^{[\leq M]}$, $n^{[\geq M]}$, and $k_{\geq M}(n)$ were defined in Section~\ref{sec:Ostr_qT-n}.

\begin{lemma}\label{lem:m_eq_qM}
Let $2 \leq m = q_M \leq n$, and assume $n^{[\leq M-1]} > 0$.
Then $\fleft,\fright$ are bijective if and only if $k_{\geq M}(n) \equiv M \pmod{2}$. 
\end{lemma}
\begin{proof}
Note that $M\geq 1$.
Choose $\absside{\cdot} \in \{\absleft{\cdot}, \absright{\cdot}\}$ so that $\absside{n^{[\geq M]}\alpha}<1/2$.
We want to show that $\xstar < n$ because then we can finish with Lemma~\ref{lem:xstar-parity-bij}. In fact, we want to show that $\xstar \leq n^{[\geq M]}$.

First, note that $n^{[\geq M]} = n - n^{[\leq M-1]} \in [n-m+1,n-1]$ because $0 < n^{[\leq M-1]}<q_M = m$. 
Therefore, we want to show that for all $x \in [n^{[\geq M]}+1, n+m-1]$ we have $\absside{x\alpha}> \absside{n^{[\geq M]}\alpha}$.

Assume first that $k_0(n^{[\geq M]})\geq M+1$. Then Lemma~\ref{lem:absside_above} tells us that
$n^{[\geq M]}$ minimises $\absside{x\alpha}$ in the range $[n^{[\geq M]}, n^{[\geq M]} + q_{M} + q_{M+1} - 1]$.
The upper bound of this range is
\begin{align*}
    n^{[\geq M]} + q_{M} + q_{M+1} - 1
    \geq n-m+ 1 + q_{M} + q_{M+1} - 1
    = n + q_{M+1}
    > n+m-1.
\end{align*}
Therefore, we indeed have $\xstar \leq n^{[\geq M]}$.

Now assume $k_0(n^{[\geq M]}) =  M$.
By the same argument as before, Lemma~\ref{lem:absside_above} tells us that
$n^{[\geq M]}$ minimises $\absside{x\alpha}$ in the range $[n^{[\geq M]}, n^{[\geq M]} + q_{M-1} + q_{M} - 1]$.
If
\begin{equation}\label{eq:okcase}
    n^{[\geq M]} + q_{M-1} + q_{M} - 1
    \geq n+m-1,
\end{equation}
then, as before, we know that $n^{[\geq M]}$ minimises $\absside{x\alpha}$ in the range $[n^{[\geq M]}, n+m-1]$, and thus $\xstar \leq n^{[\geq M]}$.
If the inequality \eqref{eq:okcase} does not hold, then, after cancelling $q_M-1=m-1$, we get
\[
    n^{[\geq M]}+ q_{M-1} < n.
\]
Since $n^{[\leq M-1]} = n - n^{[\geq M]}$, this implies $n^{[\leq M-1]} > q_{M-1}$, and so by the property \eqref{eq:greedy} of Ostrowski representations, we have $b_{M-1}(n)>0$.
But then, by the digit rules for Ostrowski representations, we must have $b_M(n)\leq a_{M+1}-1$, and so we can use the stronger statement in Lemma~\ref{lem:absside_above}: In this case $n^{[\geq M]}$ minimises $\absside{x\alpha}$ in the range $[n^{[\geq M]}, n^{[\geq M]} + q_{M-1} + 2q_{M} - 1]$.
This now covers the full range $[n^{[\geq M]}+1, n+m-1]$, and thus $\xstar \leq n^{[\geq M]}$.

Overall, we have proven that $\xstar \leq n^{[\geq M]} < n$, and so Lemma~\ref{lem:xstar-parity-bij} says that $\f$ is bijective if and only if $\absside{q_M\alpha}< 1/2$, which is equivalent to $k_{\geq M}(n) \equiv M \pmod{2}$ by Lemma~\ref{lem:absside_parity}.
\end{proof}

\begin{lemma}\label{lem:m_eq_semikonv}
Let $m = q_{M-2} + a q_{M-1}$ with $1 \leq a \leq a_{M}-1$ and $M-2\geq 0$. Let $n\geq m$ and assume that $n$ is not of the shape \ref{it:case_n_split} or \ref{it:case_n_split2} from Lemma~\ref{lem:xstarequalsn_repr} (or, in other words, $n \notin \{\xleftstar, \xrightstar \}$).
Then $\fleft,\fright$ are bijective if and only if $k_{\geq M-1}(n) \equiv M \pmod{2}$. 
\end{lemma}
\begin{proof}
Let $m,n$ be as in the lemma. We distinguish between two cases according to whether $q_{M-1}$ shows up in the representation of $n$ or not.

\caseI{1}
$n = n^{[\leq M-2]} + n^{[\geq M]}$.

Since $n^{[\leq M-2]} < q_{M-1}<m$, 
we have $n^{[\geq M]} \in [n-m+1, n+m-1]$.

Moreover, since $m\leq q_M - q_{M-1}$ and $n^{[\geq M]} > n-q_{M-1}$,
we have $[n-m+1, n+m-1]\subseteq [n^{[\geq M]} - q_M + 1, n^{[\geq M]} + q_M - 1]$.
Thus, Lemma~\ref{lem:absside_qM} implies that $n^{[\geq M]} = \xstar$ for the appropriate choice of $\xstar \in \{\xleftstar, \xrightstar\}$.
In particular, we have $\xstar < n$. 
Now Lemma~\ref{lem:xstar-parity-bij} says that $\f$ is bijective if and only if
$\absside{q_M\alpha}< 1/2$, which is equivalent to $k_{\geq M-1}(n) = k_0(n^{[\geq M]}) \equiv M \pmod{2}$, by Lemma~\ref{lem:absside_parity}.

\caseI{2}
$n = n^{[\leq M-2]} + b_{M-1} q_{M-1} + n^{[\geq M]}$ with $1\leq b_{M-1} \leq a_M$.

We want to show that $\fleft,\fright$ are not bijective.

Our first goal is to show that we may assume that one of $\xleftstar, \xrightstar$ is smaller than $n$.
If $\xleftstar<n$, this is of course the case.
Assume now that $\xleftstar >n$. Then Lemma~\ref{lem:-n_xstar} says that if we consider $q_T - n$ instead of $n$ for sufficiently large $T$, 
we get $\xrightstar(m, q_T-n)<q_T - n$. 
Lemma~\ref{lem:qT-intervals} guarantees that the balancedness property doesn't change if we replace $n$ by $q_T - n$.
Lemma~\ref{lem:flip-appl} implies (note that by assumption $n$ is not of the shape \ref{it:case_n_split2}) that $k_{\geq M-1}(q_T - n) \equiv M-1 \pmod{2}$.
If $q_{M-1}$ does not show up in the representation of $q_T - n$, then we know from Case 1 that we are in the not balanced situation.
If $q_{M-1}$ shows up in the representation of $q_T - n$, then $q_T - n$ has the same shape (the shape of Case 2) as $n$, and we can replace $q_T - n$ by $n$, knowing that now the new $\xrightstar$ is smaller than the new $n$.

Overall, in Case 2, we may now assume that at least one of $\xleftstar, \xrightstar$ is smaller than $n$. We fix this $\xstar < n$, and our goal is to show that $\f$ is not bijective. 

In order to apply Lemma~\ref{lem:xstar-parity-bij} we use the fact that $\absside{q_M\alpha}<1/2$ if and only if $k_0(\xstar) \equiv M \pmod{2}$. This is guaranteed by Lemma~\ref{lem:absside_parity}, unless $k_0(\xstar) = 0$.
In the exceptional case we still have $\absside{q_0\alpha}<1/2$. (To see this, assume the contrary and use Lemma~\ref{lem:q0-wrongside-improve} and the assumption $\xstar < n$ to obtain a contradiction.)

Now, if  $k_0(\xstar) \equiv M+1 \pmod{2}$, i.e., $\absside{q_{M-1}\alpha}<1/2$, then 
Lemma~\ref{lem:xstar-parity-bij} says that $\f$ is not bijective, as desired.

If $k_0(\xstar) \equiv M \pmod{2}$, i.e., $\absside{q_{M}\alpha}<1/2$, and if $\xstar < n - q_{M-1}$, then Lemma~\ref{lem:xstar-parity-bij} again says that $\f$ is not bijective.

We are left with the case  $n - q_{M-1} \leq \xstar < n$ and $k_0(\xstar) \equiv M \pmod{2}$.
We need to show that this is impossible for $n = n^{[\leq M-2]} + b_{M-1} q_{M-1} + n^{[\geq M]}$ with $1\leq b_{M-1} \leq a_M$ (and $n$ not of the shape \ref{it:case_n_split2}).
The shape of $n$ and the assumptions $n - q_{M-1} \leq \xstar < n$ and $k_0(\xstar) \equiv M \pmod{2}$ imply that  $k_0(\xstar)\geq M$ is impossible. Thus, we must have $k_0(\xstar)\leq M-2$.
But then, $\absside{(\xstar+q_{M-1})\alpha} < \absside{\xstar \alpha}$ by Lemma~\ref{lem:add_next_q}. Since $\xstar + q_{M-1} \in [n - m+1, n+m-1]$, this contradicts the fact that $\xstar$ minimises $\absside{x\alpha}$ in the range $[n - m+1, n+m-1]$. 
\end{proof}

\subsection{Very small $m$}\label{sec:very-small-m}

In the previous two subsections we assumed $m\geq q_1$ for technical reasons.
The cases where $2 \leq m < q_1$ are actually quite easy and we deal with them by going back to the original definition of balanced intervals.

\begin{lemma}\label{lem:m-very-small}
Let $2\leq m\leq n$ with $m < q_1$.
Then the intervals of length $\fract{n\alpha}$ are balanced with respect to $\{ 0, \alpha, \ldots, (m-1)\alpha\}$
if and only if either $n$ is either of the shape $n = \sum_{k=1}^N b_k q_k$ or $n = q_0 + \sum_{k=1+2t}^N b_k q_k$ with $t\geq 0$ and $b_{1+2t}\neq 0$.
\end{lemma}
\begin{proof}
From the basic properties of continued fractions we know that $\alpha < 1/q_1$, and so we have $0<\alpha < 2\alpha < \dots < (q_1-1)\alpha < q_1\alpha < 1$.

If $n$ is of one of the two special shapes from the lemma, then we know from Lemma~\ref{lem:Ost_largedigits}(\ref{it:dist_ub}, \ref{it:dist_specialsmall}) that $\dist{n\alpha} < \alpha$, and so the intervals of length $\dist{n\alpha}$ (which are either exactly the intervals of length $\fract{n\alpha}$ or their complements) each contain either no points or one point. In particular, the intervals are balanced.

If $n$ is not of one of the two special shapes, then we know from Lemma~\ref{lem:Ost_largedigits}(\ref{it:dist_specialsmall}) that $1/2 > \dist{n\alpha} > \alpha$, and so $\fract{n\alpha} > \alpha$. 

By simply counting points, we see that the interval $[0,\fract{n\alpha})$ contains exactly $\ceil{\fract{n\alpha}/\alpha} \geq 2$ points.

On the other hand, since $(q_1-1)\alpha < q_1\alpha < 1$ and $m-1 \leq q_1 - 2$, we see that the interval $[1-\fract{n\alpha},0)$ contains at most $\ceil{\fract{n\alpha}/\alpha}-2$ points.
Thus, the intervals are not balanced.
\end{proof}

\subsection{Finishing the proof of the full characterisation} 

We have now essentially proved Theorem~\ref{thm:bal_char_sturmian}, and we summarise the arguments below.

\begin{proof}[Proof of Theorem ~\ref{thm:bal_char_sturmian}]
Recall the four cases from Theorem ~\ref{thm:bal_char_sturmian}:
\begin{enumerate}[label = (\roman*)]
    \item\label{it:split1} 
    $m = \sum_{k = 0}^M b_k q_k$ and 
    $n = \sum_{k = M+1}^N b_k q_k$;
    \item\label{it:split2} 
    $m = \sum_{k = 0}^M b_k q_k$ with $b_M \neq 0$, and 
    $n = q_M + \sum_{k = M + 1 + 2t}^N b_k q_k$ with $t\geq 0$ and $b_{M+1+2t} \neq 0$.
    \item\label{it:conv} 
    $m = q_M$ and $n = \sum_{k = 0}^{M-1} b_k q_k + \sum_{k = M+2t}^{N} b_k q_k$ with $t\geq 0$ and $b_{M+2t}\neq 0$;
    \item\label{it:semiconv} 
    $m = q_{M-1} + a q_{M} $ with $1 \leq a \leq a_{M+1}-1$ and 
    $n = \sum_{k = 0}^{M-1} b_k q_k + \sum_{k = M+1+2t}^{N} b_k q_k$ with $t\geq 0$ and $b_{M+1+2t}\neq 0$.
\end{enumerate}
Note that some of the cases overlap; for example if the small parts of $n$ in \ref{it:conv} or \ref{it:semiconv} are zero, then we are also in the case \ref{it:split1}. (This was done for readability in Theorem~\ref{thm:bal_char_sturmian}.)
We need to show that $m,n$ are of the shape of at least one of the cases if and only if the $m\times n$ rectangles of the Sturmian words with slope $\alpha$ are balanced. 
By Theorem~\ref{thm:def-equiv} and Lemma~\ref{lem:bal_fg}, this is equivalent to $\fleft, \fright$ being bijective (and we know that one function is bijective if and only if the other is bijective). 

We first check that if $m,n$ are of the shape of one of the cases \ref{it:split1}--\ref{it:semiconv}, then $\fleft, \fright$ are bijective.
In the cases \ref{it:split1}, \ref{it:split2} we have $m < q_{M+1}$, 
and so Lemma~\ref{lem:xstarequalsn_repr} implies that $n\in \{\xleftstar, \xrightstar\}$, and Lemma~\ref{lem:xstarequalsn_bij} implies that the corresponding function $\f$ is bijective.
If we are in the case \ref{it:conv} and $n^{[\leq M-1]}=0$, then again Lemmas~\ref{lem:xstarequalsn_repr} and \ref{lem:xstarequalsn_bij} imply that $\fleft,\fright$ are bijective. 
If we are in the case \ref{it:conv} and $n^{[\leq M-1]} > 0$, this is provided by Lemma~\ref{lem:m_eq_qM}.
If we are in the case \ref{it:semiconv} and not in one of the previous cases, then Lemma~\ref{lem:m_eq_semikonv} says that $\fleft,\fright$ are bijective.

For the implication in the other direction, assume that $2\leq m \leq n$ and that $\fleft,\fright$ are bijective.

If $n\in \{\xleftstar, \xrightstar\}$, then Lemma~\ref{lem:xstarequalsn_repr} and the property \eqref{eq:greedy} imply that $m,n$ are of the shape \ref{it:split1} or \ref{it:split2}, or $m = q_M$.
If $m$ is a convergent, then Lemma~\ref{lem:m_eq_qM} implies that $m,n$ are of the shape \ref{it:conv} or \ref{it:split1}.

If $n \neq \xleftstar, \xrightstar$, then Lemma~\ref{lem:notbal_genericcase} implies that $m$ must either be a convergent or a semi convergent, or $m<q_1$. If $m$ is a convergent, then as before Lemma~\ref{lem:m_eq_qM} implies that $m,n$ are of the shape \ref{it:conv} or \ref{it:split1}. If $m$ is a semi-convergent, then Lemma~\ref{lem:m_eq_semikonv} implies that $m,n$ are of the shape \ref{it:semiconv} or \ref{it:split1} or \ref{it:split2}.
Finally, if $m<q_1$, Lemma~\ref{lem:m-very-small} implies that $m,n$ are of the shape \ref{it:split1} or \ref{it:split2}.
\end{proof}

\section{Acknowledgements}

I want to thank Jeffrey Shallit for helpful discussions and for providing several automata which hugely helped to guess the characterisation in Theorem~\ref{thm:bal_char_sturmian}.
I am also grateful to Manuel Hauke for helpful discussions, and to Benjamin Ward and Victor Beresnevich for their support in moments of despair.

\bibliographystyle{habbrv}
\bibliography{refs}

\end{document}